\documentclass[twoside,10pt,a4paper]{amsart}

\usepackage[T1]{fontenc} 
\usepackage[ansinew]{inputenc}

\usepackage{amsthm}
\usepackage{amssymb}
\usepackage{amsmath}

\usepackage{graphics}

\usepackage[all]{xy}

\usepackage{hyperref}

\DeclareMathOperator{\id}{id}
\DeclareMathOperator{\Hom}{Hom}
\DeclareMathOperator{\Ext}{Ext}
\DeclareMathOperator{\tate}{\widehat{Ext}}
\DeclareMathOperator{\Sq}{Sq}
\DeclareMathOperator{\norm}{norm}
\DeclareMathOperator{\res}{res}
\newcommand{\smatrix}[1]{\left(\!\begin{smallmatrix} #1\end{smallmatrix}\!\right)}

\newtheorem{lemma}[equation]{Proposition}

\newtheorem{thm}[equation]{Theorem}
\newtheorem{cor}[equation]{Corollary} 

\theoremstyle{remark}
\newtheorem{remark}[equation]{Remark}
\newtheorem{example}[equation]{Example}

\DeclareMathOperator{\tr}{tr}
\DeclareMathOperator{\Pp}{P}

\newcommand{\A}{\mathcal{A}}
\newcommand{\B}{\mathcal{B}}
\newcommand{\C}{\mathcal{C}}
\newcommand{\E}{\mathcal{E}}
\newcommand{\F}{\mathcal{F}}
\newcommand{\Ff}{\mathbb{F}}
\newcommand{\K}{\mathcal{K}}
\newcommand{\Zz}{\mathbb{Z}}
\newcommand{\rem}[1]{}

\newcommand{\usmod}{\text{$\mathfrak{mod}$-}}
\newcommand{\stmod}{\text{$\underline{\mathfrak{mod}}$-}}

\newcommand{\ob}[1]{\smash[t]{\overline{#1}}}

\numberwithin{equation}{section}
\begin{document}
\SelectTips{cm}{} 

\title{Dyer-Lashof operations on Tate cohomology\\ of finite groups}
\author{Martin Langer}
\email{martinlanger@uni-muenster.de}
\address{Westf\"alische Wilhelms-Universit\"at M\"unster,
         Institut f\"ur Mathematik,
         Einsteinstr.~62,
         48149~M\"unster,
         Germany}

\begin{abstract}
Let $k=\Ff_p$ be the field with $p>0$ elements, and let $G$ be a finite group. By exhibiting an $E_\infty$-operad action on $\Hom(P,k)$ for a complete projective resolution $P$ of the trivial $kG$-module $k$, we obtain power operations of Dyer-Lashof type on Tate cohomology $\hat{H}^*(G; k)$. Our operations agree with the usual Steenrod operations on ordinary cohomology $H^*(G)$. We show that they are compatible (in a suitable sense) with products of groups, and (in certain cases) with the Evens norm map. These theorems provide tools for explicit computations of the operations for small groups $G$. We also show that the operations in negative degree are non-trivial.

As an application, we prove that at the prime $2$ these operations can be used to determine whether a Tate cohomology class is productive (in the sense of Carlson) or not. 
\end{abstract}
\maketitle

\section{Introduction}
Let $k = \mathbb{F}_p$ be the field with $p$ elements. For every finite group $G$, let $\hat{H}^*(G) = \hat{H}^*(G,k)$ denote the graded Tate cohomology algebra of $G$ over $k$. Then $\hat{H}^*$ is functorial with respect to injective group homomorphisms. The starting point of our discussion will be the following Theorem.
\begin{thm}\label{maintheorem}
There is a family of $k$-linear operations $Q_s$ (and $\beta Q_s$ for $p\geq 3$) for all integers~$s$ on Tate cohomology $\hat{H}^*$, satisfying the following properties.
\begin{enumerate}
 \item The operations $Q_s$ are natural with respect to injective group homomorphisms.
 \item The operation $Q_s$ lowers the degree by $2s(p-1)$ (by $s$ if $p=2$), and $\beta Q_s$ lowers the degree by $2s(p-1)-1$ for $p>2$.
 \item $Q_s(x)=0$ if $2s < -(p-1)|x|$ (if $s < -|x|$ for $p=2$).
 \item If $p>2$, then $\beta Q_s(x)=0$ if $2s\leq -(p-1)|x|$.
 \item $Q_s(x)=x^p$ if $2s = - (p-1)|x|$ (if $s = -|x|$ for $p=2$).
 \item $Q_s(1)=0$ unless $s\neq 0$, where $1\in\hat{H}^0(G)$ is the unit element.
 \item The internal Cartan formula holds:
   \begin{align*}
     Q_s(xy) &= \sum_{i+j=s} Q_i(x)Q_j(y),  \\
     \beta Q_{s}(xy) &= \sum_{i+j=s} \beta Q_{i}(x) Q_j(y) + (-1)^{|x|} Q_i(x) \beta Q_{j}(y)  & \text{for all $x,y\in\hat{H}^*(G)$.} 
   \end{align*}
 \item The Adem relations hold: For $r>ps$,
   \begin{align*}
     Q_r Q_s = \sum_i (-1)^{r+i} (pi-r,r-(p-1)s-i-1) Q_{r+s-i}Q_i 
   \end{align*}
   and for $r\geq ps$ and $p>2$
   \begin{align*}
     Q_r \beta Q_s &= \sum_i (-1)^{r+i} (pi-r,r-(p-1)s-i) \beta Q_{r+s-i} Q_i  \\ 
                   &\quad - \sum_i (-1)^{r+i}(pi-r-1,r-(p-1)s-i) Q_{r+s-i} \beta Q_i.
   \end{align*}
   Here the convention is that $(a,b)=0$ if $a$ or $b$ is negative, and $(a,b) = \binom{a+b}{b}$ otherwise.
  \item \label{maintheoremsteenrodpart} On classes of non-negative degree, the operations agree with the Steenrod operations on $H^*(BG;k)=H^*(G)$. More precisely, for $|x|\geq 0$ we have
    \begin{align*}
      Q_{-n} (x) &= \Sq^n (x) & \text{for $p=2$ and $n\geq 0$,} \\
      Q_{-n} (x) &= P^n (x), \, \beta Q_{-n} (x) = \beta P^n (x) & \text{for $p>2$ and $n\geq 0$,}  \\
      Q_{-n} (x) &= 0, \, \beta Q_{-n}(x) =0 & \text{for $n<0$.}
    \end{align*}
\end{enumerate}
\end{thm}

We define the total operation $Q = \sum_{i} Q_i$; then the Cartan formula reads $Q(xy) = Q(x)Q(y)$ for all $x,y$. We will sometimes use the notation $\Pp_i(x) = Q_{i-|x|}(x)$, so that $\Pp_i(x)=0$ for all $i<0$ and $\Pp_0(x)=x^p$.

\begin{example} \label{examplecyclic2}
Let $p=2$, and let $G=\Zz / 2 \Zz$ be the cyclic group of order $2$. We can easily compute all the operations on $\hat{H}^*(G)$ using the statements of the theorem only. It is known that $\hat{H}^*(G) \cong k[s^{\pm 1}]$ for the unique non-zero class $s$ of degree~$1$ (see \cite{CartEile}, XII.\S7). We know that $Q(s) = s + s^2$, so that $1 = Q(1) = Q(s^{-1}s) = Q(s^{-1}) (s+s^2)$. Using the fact that $Q(s^{-1}) = s^{-2} + (\text{terms of degree less than $-2$})$ we obtain
\[ Q(s^{-1}) = s^{-2} + s^{-3} + s^{-4} + \dots. \]
More generally we get for all integers $i$
\[ Q(s^i) = (s+s^2)^i = s^{2i} (s^{-1}+1)^i = \sum_{j\geq 0} \binom{i}{j} s^{2i-j}, \]
so that $Q_{j-i}(s^i) = \binom{i}{j} s^{2i-j}$ for all $j\geq 0$. Here we use the generalized binomial coefficient
\[ \binom{i}{j} = \frac{i(i-1)\dots (i-j+1)}{j!} \quad\text{for integers $i,j$ with $j\geq 0$.} \]
\end{example}

\begin{example}
Slightly more complicated, but still an immediate consequence of the theorem is the case $G=\Zz / p \Zz$ for odd primes $p$. Here $\hat{H}^*(G) \cong k[s^{\pm 1}] \otimes \Lambda(u)$, where $s$ is of degree~$2$ and $u$ is exterior of degree~$1$. Let us define $\beta Q = \sum_i \beta Q_i$; then from the topological fact $\beta(u)=s$ we get for integers $i$
\begin{align*}
  Q(s^i) &= \sum_{j\geq 0} \binom{i}{j} s^{pi-j}, & Q(s^i u) &= Q(s^i) u, \\
 \beta Q(s^i) &= 0, & \beta Q(s^i u) &= Q(s^i) s.
\end{align*}
\end{example}

\begin{example}
Let us do an example of a non-commutative group. Let $G=Q_8$ be the quaternion group with $8$ elements. Then it is known that $\hat{H}^*(G) \cong k[s^{\pm 1},x,y] / (x^2+xy+y^2, x^3)$ with degrees $|x|=|y|=1$ and $|s|=4$.
We immediately get $Q(x)=x+x^2$ and $Q(y)=y+y^2$. Every automorphism of $H^1(G)$ is realized by a group automorphism; this implies that $\Sq^1(s)=0$ and $\Sq^2(s)=0$. From the Adem relation $\Sq^3(s) = \Sq^1 \Sq^2(s)$ it then follows that  $Q(s) = s+s^2$. By the same methods as above, one easily deduces the operations on all of $\hat{H}^*(G)$.
\end{example}

\begin{remark}
 We will prove Theorem~\ref{maintheorem} by establishing an $E_\infty$-operad action on $\Hom_{kG}(P,k)$, the cochains of a complete projective resolution $P$ of the trivial $kG$-module $k$. There is another way of constructing Dyer-Lashof operations on Tate cohomology, using equivariant stable homotopy theory as follows. In the homotopy category of $G$-spectra \cite{LewisMaySteinbergerMcClure} let $k_G = H\Ff_p$ denote the Eilenberg-MacLane spectrum, regarded as a $G$-spectrum with 'trivial $G$-action'. Associated with $k_G$ there is a Tate spectrum $t = F(EG_+,k_G) \wedge \tilde{EG}$ (see \cite{GreenleesMayGeneralizedTateCohomology}) with the property that $\hat{H}^*(G; k) \cong [S,t]_G^*$ (see \cite{GreenleesRepresentingTatecohomology}). Then \cite{McClureEinftyRingStructuresTateSpectra} shows that $t$ is an $E_\infty'$-ring spectrum, that is, we have a non-equivariant operad acting equivariantly, which can be used to define power operations on $\hat{H}^*(G)$. The author does not know whether the operations defined in that topological manner agree with the rather algebraically defined operations of this paper. 
\end{remark}

From now on, assume that the order of $G$ is divisible by $p$. Let us define a graded submodule $M^*(G)$ of a shift of $\hat{H}^*(G)$ as follows:
\[ M^n(G) = \begin{cases} 
                  \hat{H}^{n-1}(G) & \text{if $n\leq 0$,} \\ 
                  0 & \text{otherwise.}
            \end{cases} \]
Then $M^*(G)$ inherits the Dyer-Lashof operations from $\hat{H}^*(G)$, because classes of negative degrees are mapped to classes of negative degrees (or to $0$) by the $Q_i$. Via the identification $M^*(G) \cong \bigl( \hat{H}^*(G) / H^*(G) \bigr) [1]$ we can also view $M^*(G)$ as a left $H^*(G)$-module. For finite groups $G_1$ and $G_2$, we have the K\"unneth isomorphism $H^*(G_1\times G_2)\cong H^*(G_1)\otimes H^*(G_2)$ which is known to be an isomorphism of modules over the Steenrod algebra. We also have the isomorphism $M^*(G_1\times G_2)\cong M^*(G_1)\otimes M^*(G_2)$ which is an isomorphism of modules over $H^*(G_1\times G_2)$. Even more is true:

\begin{thm} \label{producttheorem}
For finite groups $G_1$ and $G_2$, the K\"unneth isomorphism $M^*(G_1\times G_2)\cong M^*(G_1)\otimes M^*(G_2)$ is an isomorphism of modules over the Dyer-Lashof algebra. In other words, $Q(\alpha\otimes \beta) = Q(\alpha)\otimes Q(\beta)$ for all $\alpha\in M^*(G_1)$ and $\beta\in M^*(G_2)$.
\end{thm}

\begin{example} \label{exampleC2C2}
 Let us consider the case $G=\Zz / 2 \Zz \times \Zz / 2 \Zz$ at the prime $p=2$. Let $\varphi_i$ be a generator of $M^{-i}(\Zz / 2\Zz)$; it corresponds to $s^{-i-1}\in\hat{H}^*(\Zz / 2\Zz)$, but this notation suggests the existence of an internal product which we do not have on $M^*$. Let us write $\varphi_{ij}\in\hat{H}^{-i-j-1}(G)$ for the element $\varphi_i\otimes\varphi_j\in M^*(G)$; then $\hat{H}^*(G)$ is the commutative graded algebra generated by polynomial classes $x,y$ of degree~$1$ (coming from the two factors of $G$) and the classes $\varphi_{ij}$ subject to the relations
\begin{align*}
\varphi_{ij}x &= \begin{cases} \varphi_{i-1,j} & \text{if $i\geq 1$,} \\ 0 & \text{otherwise,} \end{cases}  \\
\varphi_{ij}y &= \begin{cases} \varphi_{i,j-1} & \text{if $j\geq 1$,} \\ 0 & \text{otherwise,} \end{cases}  \\
\varphi_{ij}\varphi_{i'j'} &= 0.
\end{align*}
The total square on $\varphi_0$ is given by $Q(\varphi_0) = \varphi_1+\varphi_2+\dots$, see Example~\ref{examplecyclic2}. By the theorem, the total square on $\varphi_0\otimes\varphi_0=\varphi_{00}\in\hat{H}^*(G)$ is given by $Q(\varphi_{00}) = \sum_{i,j\geq 1} \varphi_{ij}$. More generally we get the formula
\[ Q(\varphi_{ij}) = \sum_{k,l\geq 0} \binom{k+i}{k} \binom{l+j}{j} \varphi_{2i+k+1,2j+l+1}. \]
In particular $\Pp_0=0$ and $\Pp_1(\varphi_{ij}) = \varphi_{2i+1,2j+1}$. 
\end{example}

In the same spirit we can prove:
\begin{cor}
If a group $G$ is a direct product of $r$ groups of order divisible by $p$, then $\Pp_j$ vanishes on elements of negative degree for all $j<\frac{r-1}{2}$ (for all $j<r-1$ if $p=2$).
\end{cor}
\begin{proof}
 Let $G=G_1\times\dots\times G_r$, and take an element of the form $a = a_1\otimes\dots\otimes a_r$ with $a_i\in\hat{H}^{|a_i|}(G_i)$ and $|a_i|<0$. Then $|a| = |a_1|+\dots+|a_r| + r-1$. Now $Q(a_i)$ is a sum of elements of degrees at most $p|a_i|$, and therefore $Q(a)$ can be written as sum of elements $b = b_1\otimes\dots\otimes b_r$ with $|b_i|\leq p|a_i|$, so that 
\[ |b| = |b_1| + \dots + |b_r| + (r-1) \leq p |a_1| + \dots + p|a_r| + (r-1) = p|a| - (p-1)(r-1). \]
This implies the result.
\end{proof}

\begin{remark}
 Notice that, unlike the ordinary Steenrod operations \cite{evenssteenrodtransfer}, the operations $Q_i$ are not compatible with transfers. For instance, if we embed $K=\Zz/2\Zz\times\{0\}\subset \Zz/2\Zz\times\Zz/2\Zz=V$, then the diagram
\[
\xymatrix{ \hat{H}^{-1}(K) \ar[r]^{Q_1}_{\cong} \ar[d]_{\tr_{K,V}} & \hat{H}^{-2}(K) \ar[d]_{\tr_{K,V}}^{\neq 0} \\
           \hat{H}^{-1}(V) \ar[r]^{Q_1}_{=0} & \hat{H}^{-2}(V) }
\]
cannot commute. 
\end{remark}

\subsection{Notations and conventions}
Throughout the paper, $p$ is a prime number and $k=\Ff_p$ is the prime field of characteristic $p$. Some of the results also hold for arbitrary fields of characteristic $p$, but then certain $k$-vector spaces have to be twisted by the Frobenius map. Groups labelled $G$, $K$, $L$ are assumed to be finite. All modules are right modules, unless mentioned otherwise. We will mainly work in $\usmod kG$, the category of right $kG$-modules, with its tensor product $\otimes$ and internal Hom-object $\Hom_k$ obtained from the Hopf algebra structure on $kG$. We will use several known results about projective modules without further notice (e.g., projective is the same as injective, the tensor product of a projective and an arbitrary module is projective, and arbitrary products and sums of projectives are projective). The ground field $k$ is considered as an object in $\usmod kG$ by the trivial $G$-action. We denote by $\stmod kG$ the stable module category, obtained from $\usmod kG$ by dividing out those morphisms which factor through a projective module. Homomorphisms in $\stmod kG$ between modules $X,Y$ are denoted by $\underline{\Hom}(X,Y)$. The category $\stmod kG$ is a triangulated category with shift functor $\Sigma=\Omega^{-1}$, and Tate cohomology can be defined as $\hat{H}^n(G) = \underline{\Hom}(\Omega^n k,k)$, with the composition product as multiplication. A morphism $X\rightarrow Y$ in $\usmod kG$ is called a \emph{stable equivalence} if it induces an isomorphism in the stable category. See \cite{Carlson} for an introduction to the stable module category.

Notice that, in this paper, we use the notation $\otimes$ for the internal tensor product $\otimes_k$ of $\usmod kG$, but $\Hom$ is used for the $k$-vector space of $kG$-linear maps, that is, $\Hom = \Hom_{kG}$. Furthermore, the symbol $\partial$ is used for the differential of chain complexes over $kG$, whereas $d$ often denotes the differential of cochain complexes over $k$. 

\subsection{Plan of the paper}
In \S\ref{sectionOperad}, we will construct the $E_\infty$ operad acting on $\Hom_{kG}(P,k)$ for a projective resolution $P$ of the trivial $kG$-module $k$. We also compare the Dyer-Lashof operations obtained from that action with the usual Steenrod operations that we have on $H^*(G)\cong H^*(BG)$, thereby completing the proof of Theorem~\ref{maintheorem}. In \S\ref{productsofgroups} we prove Theorem~\ref{producttheorem} about products of groups. In \S\ref{sectionNegativeExtGroups} we give a description of negative Tate Ext-groups in terms of complexes of projective modules, a tool we need for the proofs in the later sections. The duals of certain operations are shown to commute with the Evens norm map in \S\ref{sectionEvensNorm}, where we also show that our operations are non-trivial in negative degrees. Finally, in the last section we provide a criterion (Theorem~\ref{productiveTheorem}) for Tate cohomology classes to be productive in the sense of Carlson. 

\subsection{Acknowledgements}
Most of the paper evolved from parts of my PhD thesis written at the University of Bonn under supervision of Stefan Schwede. I would like to thank him for suggesting that project, and for his interest and helpful comments on this paper.
Furthermore, I would like to thank Wolfgang L\"uck for the financial support.  

\section{The operad}\label{sectionOperad}
\subsection{Resolutions}
Let $k$ be a field of characteristic $p$, and let $G$ be a finite group. Let $M$ be a $kG$-module. A \emph{complete projective resolution} of $M$ is a long exact sequence of projective $kG$-modules
\begin{align*}
 \xymatrix{ 
    \dots & P_{-2} \ar[l] & P_{-1} \ar[l] & & P_0 \ar[ll] \ar[dl]^{\varepsilon} & P_1 \ar[l]^-{\partial_1} & P_2 \ar[l] & \dots \ar[l] \\
      & & & M \ar[ul] } 
\end{align*}
such that $\varepsilon$ is the cokernel map of $\partial_1$. The map $\varepsilon$ is called \emph{augmentation} and can be viewed as a chain map $\varepsilon: P \rightarrow M$, where $M$ is regarded as a complex concentrated in degree~$0$. If $N$ is another module, then a map $\varepsilon':P_0\rightarrow N$ (or, equivalently, a chain map $\varepsilon':P\rightarrow N$) will be called \emph{quasi-augmentation} if there is a stable equivalence $f:M\rightarrow N$ such that $f\circ\varepsilon = \varepsilon'$. 

There is a dual notion using injectives. A \emph{complete injective resolution} of $M$ is a long exact sequence of injective $kG$-modules
\begin{align*}
 \xymatrix{ 
    \dots & I_{-2} \ar[l] & I_{-1} \ar[l] & I_{0} \ar[l]^-{\partial_0} & & I_{1} \ar[ll] \ar[dl] & I_{2} \ar[l] & \dots \ar[l] \\
    &  & & & M \ar[ul]^{\eta} } 
\end{align*}
in which $\eta$ is the inclusion of the kernel of $\partial_0$. The map $\eta$ is called \emph{coaugmentation} and can be viewed as a chain map $\eta:M\rightarrow I$. If $N$ is another module, then a map $\eta':N\rightarrow I$ will be called \emph{quasi-coaugmentation} if there is a stable equivalence $f:N\rightarrow M$ with $\eta\circ f = \eta'$. 

Since projectives are the same as injectives, the notions of complete resolutions only differ in the position of the resolved module $M$. If $P$ is a complete projective resolution of the trivial module $k$, and $N$ is another $kG$-module, then the cohomology groups of the complex $\Hom_{kG}(P,N)$ define the Tate cohomology of $G$ with coefficients in $N$, that is, $\hat{H}^n(G; N) \cong H^n \Hom_{kG}(P,N)$.

\begin{lemma}
Let $P$ and $Q$ be complete projective resolutions, and let $\epsilon:P\rightarrow M$ be a quasi-augmentation. If for some chain transformation $f:Q\rightarrow P$ the composite $Q_0\xrightarrow{f_0} P_0\xrightarrow{\epsilon} M$ is zero, then $f$ is null-homotopic. The corresponding statement holds for injective resolutions.
\end{lemma}
We omit the straightforward proof.

Whenever $C$ is a cochain complex of $kG$-modules, we define the dual complex $C^\vee$ as $(C^\vee)_n = \Hom(C_{-n},k)$ with the induced differentials. If $P$ is a complete projective resolution of $M$ with (quasi-)augmentation $\epsilon$, then $P^\vee$ is a complete injective resolution of $M^\vee$ with (quasi)-coaugmentation $\epsilon^\vee$, and the same is true with the roles of projective and injective interchanged.

Let $k$ be the trivial $kG$-module, and choose complete injective resolutions $I$ and $I'$ of $k$ with coaugmentations $\eta,\eta'$. The tensor product $I\otimes I'$ is defined to be the complex with modules $(I\otimes I')_n = \bigoplus_{i+j=n} I_i\otimes I'_j$ and differential $\partial_{I\otimes I'} = \partial_I \otimes \id + \id \otimes \partial_{I'}$ (note here that evaluation of the differential involves the usual sign, i.e., $(\id\otimes\partial) (x\otimes y) = (-1)^{|x|} x\otimes \partial y$). It is known that the tensor product $I\otimes I'$ is a complete injective resolution of $k$ with quasi-coaugmentation $\eta\otimes \eta'$ (see \cite{KrauseNoetherianScheme}, \S8). 

Now let $P$ and $P'$ be complete projective resolutions of $k$, and assume that all modules $P_i,P'_i$ are finitely generated. Let us define a new tensor product
$P \boxtimes P' = (P^\vee \otimes P'^\vee)^\vee$; more explicitly, $(P \boxtimes P')_n = \prod_{i+j=n} P_i \otimes P'_j$. By the considerations above, this is a complete projective resolution of $k$ with quasi-augmentation $\varepsilon\boxtimes\varepsilon': P\boxtimes P' \rightarrow k$. These definitions and observations can be generalized to $\boxtimes$-products of finitely many complete projective resolutions.

\begin{remark}\label{deltaproduct}
 The $\boxtimes$-product can be used to define the multiplication on Tate cohomology. By usual homological algebra, the identity map on $k$ can be lifted to a commutative diagram as follows:
\[ \xymatrix{ P \ar[r]^-{\Delta} \ar[d]_{\varepsilon} & P \boxtimes P \ar[d]^{\varepsilon\boxtimes \varepsilon} \\  k \ar@{=}[r] & k } \]
Such a lift is unique up to homotopy. A more explicit construction of $\Delta$ is given in the proof of Theorem 4.1 in \cite{CartEile}, XII, where it is also shown that $\Delta$ induces the Tate cohomology product in the following way: given cycles $f,g\in\Hom^*_{kG}(P,k)$ we get a cycle $(f\boxtimes g)\circ \Delta\in\Hom^*_{kG}(P,k)$ representing $[f]\cdot [g]$. 
\end{remark}

\subsection{Motivation for the definition of the operad}
Let $P$ be a complete projective resolution of $k$ by finitely generated $kG$-modules. Before we start with the actual construction of an $E_\infty$-structure on $\Hom^*_{kG}(P,k)$, let us point out a major issue in the construction of power operations which does not turn up in the case of ordinary cohomology $H^*(G)$. For simplicity, let us assume that $p=2$ throughout this motivational part. Let us naively transfer to Tate cohomology the construction of $\Sq_1$ as it is done in ordinary cohomology. We know that the identity map of $k$ can be lifted to a map $\Delta:P\rightarrow P^{\boxtimes 2}$ as in Remark~\ref{deltaproduct}, and any two such liftings are homotopic. Therefore, if $T$ denotes the twist map of $P^{\boxtimes 2}$, then we know that $(1-T)\circ \Delta$ is the boundary of some map $\Delta_1:P\rightarrow P^{\boxtimes 2}$ of degree~$-1$. If $\zeta:P\rightarrow k$ is a chain map of degree~$n$ representing some cohomology class $[\zeta]\in\hat{H}^n(G)$, then we know that $\zeta^{\boxtimes 2}\circ \Delta_1$ is a chain map of degree~$2n$, and we could define $\Sq_1(\zeta)$ to be the class represented by that map. The problem is here that there is an ambiguity in the choice of the map $\Delta_1$, and any two such choices differ by a chain map $P\rightarrow P^{\boxtimes 2}$ of degree~$-1$. Therefore, $\Sq_1(\zeta)$ is only well-defined up to some element in $\zeta^2\cdot \hat{H}^{-1}(G)$. This problem does not occur in ordinary cohomology simply because $H^{-1}(G)$ is zero. We therefore have to rigidify our choice of $\Delta_1$ in order to get actual operations. To do so, observe that a chain map $P\rightarrow P^{\boxtimes 2}$ of degree~$-1$ certainly represents the zero class if the composite $P_{-1}\rightarrow (P^{\boxtimes 2})_0 \twoheadrightarrow P_0\otimes P_0\rightarrow k$ vanishes, so that one possibility is to require the map $P_{-1}\rightarrow P_0\otimes P_0$ to be zero. The next step is to elaborate this idea, and because we want an $E_\infty$-structure, we need to do so in an 'operadic' way. 

\subsection{Definition of the operad}\label{tateoperaddefinition}
As before, let $P$ be a complete projective resolution of $k$ by finitely generated $kG$-modules. We are now going to define an acyclic operad which acts on $\Hom^*_{kG}(P,k)$. To do so, we will work in the category of (increasing degree) differential graded modules over $k$ (or, equivalently, the category of cochain complexes of $k$-vector spaces) with its symmetric monoidal tensor product $\otimes$. Recall that if $X$ and $Y$ are chain complexes of $kG$-modules with differential $\partial$, then we get such a differential graded module $\Hom^*(X,Y)$ by defining
\[ \Hom^n(X,Y) = \prod_{j\in\Zz} \Hom_{kG}(X_{n+j},Y_j) \]
with differential $d(f) = \partial f - (-1)^n f\partial$. 

Let us recall some basics about operads; see, e.g., \cite{OperadsAlgebrasModulesMotives} for an introduction. A symmetric operad $\C$ is given by a differential graded module $\C(j)$ for every integer $j\geq 0$ together with a $\Sigma_j$-action, equivariant structure maps
\[ \C(j)\otimes \C(i_1)\otimes\dots\C(i_j) \rightarrow \C(i_1+\dots+i_j) \]
for all $j,i_1,\dots,i_j$, and a unit map $k\rightarrow \C(j)$ for each $j$; all these maps have to satisfy certain coherence diagrams. 
A typical example of such an operad is the so-called \emph{coendomorphism-operad} $\F(j) = \Hom^*(P,P^{\boxtimes j})$ for $j\geq 0$, whose structure maps are given by
\begin{align*}
  \Hom(P,P^{\boxtimes j}) \otimes \Hom(P,P^{\boxtimes i_1}) \otimes \dots \otimes \Hom(P,P^{\boxtimes i_j}) &\rightarrow \Hom(P,P^{\boxtimes (i_1+\dots+i_j)}) \\
    g \otimes f_1 \otimes \dots \otimes f_j &\mapsto (f_1\boxtimes \dots \boxtimes f_j)\bullet g.
\end{align*}
We have written $\bullet$ here because we want to stress that the Koszul sign rule also applies to this situation; whenever $a$ and $b$ are composable maps of certain degrees, we write $a\bullet b$ for $(-1)^{|a||b|} \cdot a\circ b$, so that expressions like $b\otimes a\mapsto a\bullet b$ indeed yield maps of chain complexes. The symmetric group $\Sigma_j$ acts on $P^{\boxtimes j}$ by permutation of the factors (note that this also involves the usual signs), and we therefore get an action of $\Sigma_j$ on $\Hom(P,P^{\boxtimes j})$. The unit map $k\rightarrow \Hom(P,P)$ is given by the identity of $P$. The operad we are up to will be a sub-operad of the coendomorphism-operad $\F$. 

An operad $\C$ is called \emph{unital} if $\C(0)=k$. In that case, the $\C(j)$ have augmentations coming from the operad structure maps
\[ \C(j) \cong \C(j)\otimes \C(0)^j \rightarrow \C(0) = k. \]
The operad is called \emph{acyclic} if the augmentations are quasi-isomorphisms of chain complexes. An operad $\C$ is called an \emph{$E_\infty$-operad} if it is acyclic and for every $j$, $\C(j)$ is free as a $k\Sigma_j$-module. A differential graded module $A$ is called a \emph{$\C$-algebra} if there are structure maps
\[ \C(j) \otimes A^j \rightarrow A \]
for every $j\geq 0$ which are associative, unital and equivariant (see \cite{OperadsAlgebrasModulesMotives}, \S2 for details). Our goal is to define an acyclic operad $\C$ (and later an $E_\infty$-operad) and a $\C$-algebra structure on $A=\Hom^*(P,k)$. This structure can then be used to define the operations $Q_i$ on $H^*A\cong\hat{H}^*(G)$, and also for proving most of Theorem~\ref{maintheorem}. 

Let us begin with the definition of $\C$. For every non-negative integer $j$, we define a differential graded submodule $\C(j)$ of $\Hom^*(P,P^{\boxtimes j})$ as follows:
\begin{align*}
  \C(j)^m &= 0  & \text{for $m>0$,}  \\
  \C(j)^0 &= \{ f\in\Hom^0(P,P^{\boxtimes j}) \,\mid\, df = 0 \} \\
  \C(j)^m &= \left\{ f\in\Hom^m(P,P^{\boxtimes j}) \,\biggl| \, \begin{array}{l} \text{$P_i \xrightarrow{\text{proj}\circ f} P_{s_1}\otimes P_{s_2}\otimes\dots\otimes P_{s_j}$ vanishes} \\ \text{for all $i<0$ and all $s_1,\dots,s_j\geq 0$} \end{array} \right\}  & \text{for $m<0$.}
\end{align*}
In order to check that $\C(j)$ is indeed a differential graded submodule, we have to prove $d\C(j)^m \subseteq \C(j)^{m+1}$. This is clear for $m\geq -1$, and in case $m<-1$, the map $P_i \xrightarrow{\text{proj}\circ df} P_{s_1}\otimes\dots\otimes P_{s_j}$ is the sum of $P_i\xrightarrow{\partial} P_{i-1}\xrightarrow{\text{proj}\circ f} P_{s_1}\otimes\dots\otimes P_{s_j}$ and maps $P_i\xrightarrow{\text{proj}\circ f} P_{s_1}\otimes\dots\otimes P_{s_t+1}\otimes\dots\otimes P_{s_j} \xrightarrow{\id\otimes\partial\otimes\id} P_{s_1}\otimes\dots\otimes P_{s_j}$, all of which are zero by assumption.

Next we show that $\C$ is a sub-operad of the co-endomorphism operad $\F$. In order to do so, we only need to show that it is closed under the structure maps, the $\Sigma$-action, and the unit. The latter two are immediate consequences of the definition, so let us take $g\in C(j)$, $f_i\in C(j_i)$ for $i=1,\dots,j$ and prove that $(f_1\boxtimes\dots\boxtimes f_j)\bullet g\in \C(j_1+\dots+j_j)$. If one of the chosen elements is of positive degree, then the composition is zero. If all the chosen elements are of degree zero, then they are chain transformations and so is the composition. Now we can assume that the composition is of negative degree, and we have to show that the composite
\[ P_i \xrightarrow{g} P_{s_1}\otimes \dots \otimes P_{s_j} \xrightarrow{ f_1\boxtimes \dots\boxtimes f_j } P_{t_{1,1}}\otimes\dots\otimes P_{t_{1,j_1}}\otimes \dots \otimes P_{t_{j,1}}\otimes\dots\otimes P_{t_{j,j_j}} \]
is zero for all $i<0$ and $t_{l,n}\geq 0$. If $s_l$ is negative, then $P_{s_l} \xrightarrow{ f_l } P_{t_{l,1}}\otimes\dots\otimes P_{t_{l,j_l}}$ vanishes and so does the composition. But if all the $s_l$'s are non-negative, then $g$ is zero, so we are done.

The operad $\C$ is unital, that is, $\C(0)$ is isomorphic to $k$ concentrated in degree $0$. Here we use the convention $P^{\boxtimes 0}=k$; then $\C(0)^m=0$ unless $m=0$, in which case
\[ \C(0)^0 = \{ f \in \Hom^0(P,k) \, \mid \, df = 0 \} \cong k \left< \varepsilon \right>. \]
So we get augmentations $\C(j) \cong \C(j)\otimes \C(0)^j \rightarrow \C(0)\cong k$ given by postcomposition with $\varepsilon^{\boxtimes j}$.

\subsection{Acyclicity of the operad}
We are now going to show that the augmentations $\C(j)\rightarrow k$ are quasi-isomorphisms. 

To do so, let us consider another complete projective resolution $Q$ of $k$, constructed as follows. Let us define $P^+$ to be the non-negative part of $P$, that is $P^+_n = P_n$ for $n\geq 0$ with the induced differentials. Then $k\xleftarrow{\varepsilon} P^+$ is an acyclic augmented complex, and by the K\"unneth theorem $k\xleftarrow{\varepsilon^{\otimes j}} (P^+)^{\otimes j}$ is also acyclic. Next, we define a complex $\dots \leftarrow R_{-2}\leftarrow R_{-1}\leftarrow R_0$ by setting $R_n = P_n$ for $n<0$ and $R_0 = k$, the differential $R_{-1}\leftarrow R_{0}$ being the coaugmentation of $P$. Then $R$ is acyclic, and by the K\"unneth theorem $R^{\otimes j}$ is also acyclic. Note that $(R^{\otimes j})_0=k$, so we can splice the complexes $R^{\otimes j}$ and $k\leftarrow (P^+)^{\otimes j}$ to get a complex $Q$, which then is a complete projective resolution of $k$. There is a chain map $\Phi:P^{\boxtimes j}\rightarrow Q$ which in non-negative degrees is given by projections, and in negative degrees the maps
\[ P_{s_1} \otimes P_{s_2} \otimes \dots \otimes P_{s_j} \rightarrow R_{s_1}\otimes R_{s_2} \otimes \dots \otimes R_{s_j} \]
are zero unless all the $s_i$'s are non-positive, in which case the map is the tensor product of identity maps and the augmentation $\varepsilon:P_0\rightarrow R_0=k$.

Since the composition $P^{\boxtimes j} \xrightarrow{\Phi} Q \xrightarrow{\varepsilon} k$ equals the quasi-augmentation $\varepsilon^{\boxtimes j}$, we get that $\Phi$ is a chain homotopy equivalence. Therefore, the induced map $\lambda:\Hom^*(P,P^{\boxtimes j}) \rightarrow \Hom^*(P,Q)$ is a quasi-isomorphism. Moreover, $\lambda$ is surjective because $\Phi$ is levelwise onto.

\begin{lemma} Suppose that $\lambda:A\rightarrow B$ is a surjective quasi-isomorphism of differential graded modules, and let $C\subseteq B$ be a differential graded submodule of $B$. Then the restriction $\lambda:\lambda^{-1}(C) \rightarrow C$ is a quasi-isomorphism as well.
\end{lemma}
\begin{proof}
Let us denote by $K$ the kernel of $\lambda$. Since $\lambda$ is a quasi-isomorphism, the long exact sequence in homology implies that $H^*(K)=0$. Since $K$ is also the kernel of $\lambda\mid_{\lambda^{-1}(C)}$, using the long exact sequence in homology again we get that the restriction of $\lambda$ is a quasi-isomorphism.
\end{proof}

Now the idea is to choose a dg submodule $C$ of $\Hom^*(P,Q)$ quasi-isomorphic to $k$, and such that $\lambda^{-1}(C)$ is (close to) our $\C(j)$. Define:
\begin{align*}
 C^m &= 0 & \text{for $m>0$, } \\
 C^0 &= \{ f\in \Hom^0(P,Q) \,|\, df = 0 \} \\
 C^m &= \{ f\in \Hom^m(P,Q) \,|\, \text{$P_i \xrightarrow{f} Q_j$ is zero for all $i<0\leq j$} \} & \text{for $m<0$.}
\end{align*}
Then $C$ is indeed a dg submodule of $\Hom^*(P,Q)$.

\begin{lemma}\label{cohomologyofC}
 We have $H^*(C)\cong k$.
\end{lemma}
\begin{proof}
Clearly, $H^m(C)=0$ for $m>0$. Let $m<0$, and let $f\in C^m$ be a cocycle. Define $g:P_{i+m-1}\rightarrow Q_i$ to be zero for all $i=0,1,\dots,-m$. By common homological algebra we can extend $g$ to a chain null-homotopy for $f$ (the conditions needed for the inductive construction of $g$ is that $\partial g \partial = f$ at the two boundary points of the domain on which $g$ has been defined, and this condition is clearly satisfied). Then $dg=f$ with $g\in C$, and hence $H^m(C)=0$ for $m>0$.

Finally, we claim that the image of $d:C^{-1}\rightarrow C^0$ is the same as the image of $d:\Hom^{-1}(P,Q)\rightarrow C^0$ (then it follows that $H^0(C) \cong H^0(G) \cong k$). Let $f\in\Hom^{-1}(P,Q)$; then the bottom row in the diagram
\[ \xymatrix{ P_0 \ar[rr]^{\partial} \ar[dr]_{\varepsilon}  & & P_1 \ar[r]^{f} & Q_0 \ar[dr]_(0.4){\varepsilon^{\otimes j}} \ar[rr]^{\partial} & & Q_{-1} \\
  & k \ar[rrr] \ar[ur]_{\eta} & & & k \ar[ur]  }
\]
is stably trivial and therefore the zero map (we assume here that $|G|$ is divisible by $p$, which is the only interesting case). Therefore the upper row vanishes, and by usual homological algebra there is a cocycle $g\in\Hom^{-1}(P,Q)$ with $f_0=g_0:P_{-1}\rightarrow Q_0$. Then $f-g\in C^{-1}$ and $d(f-g) = df$, so we are done.
\end{proof}

We finally use a method of chopping off the positive part of a dg module. Given a dg module $A$, define $F(A)$ to be the dg submodule given by 
\[ F(A)^m  = \begin{cases} 0 & \text{if $m>0$,} \\ \text{cycles of $A^0$} & \text{if $m=0$,} \\ A^m & \text{if $m<0$.} \end{cases} \]
(This can be viewed as the (co)connected cover of $A$.) Then the inclusion $F(A)\subseteq A$ induces an isomorphism $H^*(F(A)) \cong H^*(A)$ in non-positive degrees.

\begin{lemma}
 The augmentation $\C(j) \rightarrow k$ is a quasi-isomorphism. Thus, the operad $\C$ is acyclic.
\end{lemma}
\begin{proof}
 Note that $\C(j) = F ( \lambda^{-1}(C) )$, so that $H^*(\C(j))\cong H^*(C)\cong k$ by Proposition~\ref{cohomologyofC}. Since there is a cocycle $f\in \C(j)^0$ such that $P\xrightarrow{f} P^{\boxtimes j} \xrightarrow{\varepsilon^{\boxtimes j}} k$ equals the augmentation $\varepsilon$, the map $\C(j)\rightarrow k$ is onto in $H^0$ and therefore a quasi-isomorphism.
\end{proof}

For every operad $\A$, the module $\A(0)$ is an algebra over $\A$ via the action map $\A(j)\otimes \A(0)^{\otimes j}\rightarrow \A(0)$. In particular, $\Hom^*(P,k)$ is an algebra over the co-endomorphism operad $\Hom^*(P,P^{\otimes j})$, and we can restrict the operad action to the sub-operad $\C$. Hence, $\Hom^*(P,k)$ is a $\C$-algebra.

\begin{lemma}\label{operadRightProduct}
 The operad $\C$ induces an $E_\infty$-structure on $\Hom^*(P,k)$ in such a way that the product on $H^*\Hom^*(P,k)$ agrees with the composition product of the Tate cohomology ring $\hat{H}^*(G)$.
\end{lemma}
\begin{proof}
 The operad $\C$ might itself not be $\Sigma$-free, so we have to choose an approximation of $\C$ by an $E_\infty$-operad. One possible way of doing so is to choose an arbitrary $E_\infty$-operad $\E'$ and tensor its augmentation $\E'\rightarrow k$ with $\C$. Then $\E=\C\otimes \E'$ is an $E_\infty$-operad acting on $\Hom^*(P,k)$ via the action of $\C$ pulled back along the morphism of operads $\E\rightarrow \C$. The statement about the product follows from the fact that the element $\Delta\in\C(2)^0\subseteq \Hom^0(P,P\boxtimes P)$ given in Remark~\ref{deltaproduct} generates the cohomology $H^0(\C(2))$ and induces the right product on $H^*\Hom^*(P,k)$.
\end{proof}

\subsection{Comparison with Steenrod reduced powers}
For the proof of part \eqref{maintheoremsteenrodpart} of Theorem~\ref{maintheorem} we need to recall the construction of Steenrod operations in the cohomology of cocommutative Hopf algebras. Let $\widetilde{P}$ be an ordinary projective resolution of $k$, viewed as a complex $\dots\leftarrow 0\leftarrow P_0\leftarrow P_1\leftarrow\dots$. Then $\widetilde{P}^{\otimes j}$ is a projective resolution of $k$ for all $j$. Consider the suboperad $\A(j)=F(\Hom(\widetilde{P},\widetilde{P}^{\otimes j}))$ of the coendomorphism-operad $\Hom(\widetilde{P},\widetilde{P}^{\otimes j})$. Then $\A$ is acyclic, and $\Hom(\widetilde{P},k)$ is an $\A$-algebra in the obvious way. Using an $E_\infty$-approximation of $\A$, this operad action defines the Steenrod operations on $H^*\Hom^*(\widetilde{P},k)\cong H^*(G)$. 

Extend $\widetilde{P}$ to a complete projective resolution $P$ of $k$. We are now going to write down a quasi-isomorphism of unital operads $\C\rightarrow \A$. Let us begin with a function $\Psi$ which maps an element $f\in\Hom^*(P,P^{\boxtimes j})$ to the element in $\Hom^*(\widetilde{P},\widetilde{P}^{\otimes j})$ given by the composition
\[ \widetilde{P} \xrightarrow{\iota} P \xrightarrow{f} P^{\boxtimes j} \xrightarrow{\pi^{\boxtimes j}} \widetilde{P}^{\otimes j}. \]
Notice here that the inclusion map $\iota$ is not quite a chain map; its differential $d\iota$ in $\Hom(\widetilde{P},P)$ is zero everywhere except for $\widetilde{P}_0\rightarrow P_{-1}$. On the other hand, the projection map $\pi$ is a chain map, and therefore
\[ d (\pi f \iota) = \pi d(f) \iota \pm \pi f d(\iota) \]
in $\Hom^*(\widetilde{P},\widetilde{P}^{\otimes j})$. Now assume that $f\in \C(j)$; then either $f$ is of non-negative degree, in which case $\pi f d(\iota)$ is zero (because $\pi$ vanishes in negative degrees), or $f$ is of negative degree, but then $f$ is zero as maps $P_{-1}\rightarrow P_{s_1}\otimes P_{s_2}\otimes\dots\otimes P_{s_j}$ for all $s_i\geq 0$, and $\pi$ is zero on all other factors of $P^{\boxtimes j}$ of interest. Hence $ d(\Psi(f)) = \Psi(df)$, so that $\Psi$ restricted to $\C(j)$ is indeed a map of dg modules. We get a map $\Psi:\C\rightarrow \A$ of unital operads, and we need to show that $\Psi$ commutes with the augmentations of $\C(j)$ and $\A(j)$. This follows from the following commutative diagrams:
\[
 \xymatrix{ \tilde{P} \ar[r]^{\iota} \ar[dr]_{\varepsilon} & P \ar[d]^{\varepsilon}  & & P^{\boxtimes j} \ar[dr]_{\varepsilon^{\boxtimes j}} \ar[r]^{\pi^{\boxtimes j}} & \widetilde{P}^{\otimes j} \ar[d]^{\varepsilon^{\otimes j}} \\
 & k & & & k }
\]

\begin{proof}[Proof of Theorem~\ref{maintheorem}]
Everything except part \eqref{maintheoremsteenrodpart} is a consequence of Proposition~\ref{operadRightProduct} and the fact that $E_\infty$-structures can be used to construct power operations with the desired properties; see, e.g., I.\S7 in \cite{OperadsAlgebrasModulesMotives}, I.\S1 in \cite{HomologyIteratedLoopSpaces}, and \cite{MayGeneralAlgebraic}. For part \eqref{maintheoremsteenrodpart}, note that by construction of the operations $Q_i$ via $\C$ and the Steenrod operations via $\A$ we get the desired statement for $n\geq 0$. To prove $Q_{-n} (x)= 0$ and $\beta Q_{-n}(x)=0$ for $n<0$ it is enough to notice that for elements $f$ in $\C(p)$ we have that $f:P_{\text{neg}} \rightarrow P_{|x|}^{\otimes p}$ vanishes.
\end{proof}

\section{Products of groups}\label{productsofgroups}
This section is devoted to the proof of Theorem~\ref{producttheorem}. Let $G$ be any finite group whose order is divisible by $p$. As a first step, we shall define a new operad action defining some power operations on $M^*(G)$. In the second step we prove that these operations agree with the Dyer-Lashof operations coming from $\hat{H}^*(G)$. 

Let $P$ be a complete projective resolution of $k$ as a $kG$-module. We denote by $\ob{P}$ the complex $\dots\leftarrow P_{-2}\leftarrow P_{-1} \leftarrow 0 \leftarrow 0 \leftarrow \dots$, with the $P_{-1}$ sitting in degree $0$ and with differential $\partial_{\ob{P}} = - \partial_P$, and let $\eta:k\rightarrow \ob{P}$ be the coaugmentation. For $j\geq 1$ define the differential graded module $\B(j) = F (\Hom^*(\ob{P},\ob{P}^{\otimes j}))$. Also put $\B(0)=k$; we want to turn $\B$ into a unital operad, so we need to define the structure maps
\[ \gamma: \B(j) \otimes \B(i_1)\otimes\dots\otimes \B(i_j) \rightarrow \B(i_1+\dots+i_j). \]
As long as all $i_s$'s are positive, we simply take the usual structure maps of the coendomorphism-operad $\Hom^*(\ob{P},\ob{P}^{\otimes j})$. If one of the $i_s$'s is zero, then we put $\gamma=0$ unless $i_1=i_2=\dots=i_j=0$, in which case
\begin{align} \label{theBaugmentation} 
 \gamma: \B(j) \otimes \B(0)^{\otimes j} = \B(j) \rightarrow \Hom_{kG}(k,k) = k = \B(0) 
\end{align}
sends a chain map $\ob{P}\rightarrow \ob{P}^{\otimes j}$ in $\B(j)$ to the induced map $k\rightarrow k^{\otimes j}=k$ on zero-cycles. It is now straightforward to check that $\B$ is indeed a unital symmetric operad. Also, $\B$ is acyclic because by usual homological algebra the augmentations $\B(j)\rightarrow k$ are quasi-isomorphisms.

Now $\Hom^*(\ob{P},k)$ is a $\B$-algebra, so we get Dyer-Lashof operations on $H^* \Hom^*(\ob{P},k) \cong M^*(G)$ which we are now going to compare with those obtained from $\C$. Let $\iota\in\Hom^1(\ob{P},P)$ be the inclusion, and let $\pi\in\Hom^{-1}(P,\ob{P})$ be the projection map. Then $d\iota=0$, but $d\pi\neq 0$. Let $K$ be the cochain complex of $k$-vector spaces generated by an element $x$ of degree~$-1$ which is mapped by the differential to a non-trivial element $y$ in degree $0$:
\[ 
 \xymatrix@R=0pt{
   \dots & 0 \ar[l] & k\left<y\right> \ar[l] & k\left<x\right> \ar[l] & 0 \ar[l] & \dots \ar[l] \\
         &          &              y         & x \ar@{|->}[l] }
\]
Let $Y=K^{\otimes p}$, and then define the augmented cochain complex $X$ by the formula $X_i = Y_{i-1}$ for all $i\leq 0$ with augmentation $X_0\rightarrow Y_0=k\left<y^p\right>$. Then $X$ is an acyclic augmented complex of $k\Sigma_p$-modules. The map of cochain complexes $K\rightarrow \Hom^*(P,\ob{P})$ given by $x\mapsto \pi$ induces a map of cochain complexes $\varphi:Y\rightarrow \Hom^*(P^{\boxtimes p},\ob{P}^{\otimes p})$. Let us define $\sigma:X\otimes \C(p)\rightarrow \B(p)$ by the formula $\sigma(v \otimes f) = (-1)^{|f|}\varphi(v) \circ f \circ \iota$ (the sign coming from shifting $Y$ to $X$). 

\begin{lemma}\label{sigmalemma}
The map $\sigma:X\otimes \C(p) \rightarrow \B(p)$ enjoys the following properties:
\begin{itemize}
 \item[(a)] it is a $\Sigma_p$-equivariant cochain map lifting the identity of $k$,
 \item[(b)] for every $f\in\C(p)$ we have $\sigma(x^p \otimes f) = (-1)^{|f|} \pi^{\boxtimes p} \circ f \circ \iota$, and
 \item[(c)] for every element $w\in X\otimes \C(p)$ of bidegree $(m,n)$ with $m>-p+1$, and for every cocycle $a\in \Hom^*(\ob{P},k)$, we have $\gamma(\sigma(w)\otimes a^{\otimes p})=0$, where $\gamma$ is the operad action of $\B$ on $\Hom^*(\ob{P},k)$.
\end{itemize}
\end{lemma}

\begin{proof}
 (b) follows directly from the definition. 
To show (c), let $w=w_1\otimes\dots\otimes w_p\otimes f\in X\otimes \C(p)$ with $w_i\in\{x,y\}$ for all $i$. Up to a sign, $\gamma(\sigma(w)\otimes a^{\otimes p})$ is given by the composition
\[ \ob{P} \xrightarrow{\iota} P \xrightarrow{f} P^{\boxtimes p} \xrightarrow{ u_1\boxtimes \dots \boxtimes u_p } \ob{P}^{\otimes p} \xrightarrow{a^{\otimes p}} k, \]
where $u_i=\pi$ if $w_i=x$ and $u_i=d\pi$ if $w_i=y$. From the condition on the bidegree of $w$ we know that at least one of the $u_i$'s equals $d\pi$, so that $a\circ u_i=0$, which implies (c).

For (a), let $f\in\C(p)$ be a cocycle in degree $0$ mapping to $1$ under the augmentation $\C(p)\rightarrow k$; then consider the following diagram
\begin{align*}
 \xymatrix{ k \ar[rrr]^{\eta} \ar[ddd]_{\eta^{\otimes p}} & & & P_{-1} \ar[ddd]^{f} \\
 & & P_0 \ar[ull]_{\varepsilon} \ar[ur]^{\partial} \ar[dl]^{f} \\
 & (P^{\boxtimes p})_0 \ar[uul]_{\varepsilon^{\boxtimes p}} \ar[dl]^(0.4){-(d\pi)^{\boxtimes p}} \ar[drr]^{\partial} \\
 (\ob{P}^{\otimes p})_0 & & & \ar[lll]^{\pi\boxtimes(d\pi)^{\boxtimes (p-1)}} (P^{\boxtimes p})_{-1} } 
\end{align*}
All smaller parts commute, and since $P_0\xrightarrow{\varepsilon} k$ is surjective we can deduce that the exterior square commutes. Therefore, $\sigma$ indeed lifts the identity of $k$. Also, $\sigma$ is $\Sigma_p$-equivariant and is a cochain map because $d\iota=0$, so (a) is proved.
\end{proof}

\begin{lemma}
 The operad action of $\B$ on $\Hom^*(\ob{P},k)$ and the action of $\C$ on $\Hom^*(P,k)$ define the same operations on $M^*(G)$.
\end{lemma}
\begin{proof}
Let $\epsilon:X\rightarrow k[1-p]$ be the $k\Sigma_p$-linear chain map given by $x^p \mapsto 1$, and define $\tau = \epsilon\otimes\id_\C:X\otimes \C(p)\rightarrow \C(p)[1-p]$. By suitably shifting the action of $\C(p)$ on the negative part of $\Hom^*(P,k)$ we get a map defined by
\begin{align*}
 \gamma_\C:\C(p)[1-p]\otimes\Hom^*(\ob{P},k)^{\otimes p} &\rightarrow \Hom^*(\ob{P},k) \\
  f \otimes w &\mapsto (-1)^{|f|\cdot (|w|+1)} w\circ \pi^{\boxtimes p} \circ f \circ \iota 
\end{align*}
for all $f\in\C(p)[1-p]$ and $w\in\Hom^*(\ob{P},k)^{\otimes p}$. The sign is due to the Koszul sign rule, and the check that this is indeed a map of chain complexes uses the fact that $a\circ (d\pi)=0$ for all $a\in\Hom(P,k)$. Now $\gamma_\C$ can be used to construct the power operations on $M^*(G)$ as follows. Let $W$ be the standard free resolution of the trivial $k C_p$-module $k$ (where $C_p$ denotes the cyclic group of order $p$), so that $W_i$ is generated by a single element $e_i$. Since $X\otimes \C(p)$ is an acyclic augmented complex of $kC_p$-modules, we can lift the identity of $k$ to a $C_p$-equivariant chain map $\vartheta:W\rightarrow X\otimes \C(p)$. We then have a diagram like this:
\[
 \xymatrix@R=3pt{ 
   & & \B(p) \\
   W \ar[r]^-{\vartheta} & X\otimes \C(p) \ar[ur]^-{\sigma} \ar[dr]_-{\tau} \\
   & & \C(p)[1-p] } 
\]
For cocycles $a\in\Hom^*(\ob{P},k)$, define $D_i^\C(a)$ to be the cohomology class of the cocycle $\gamma_\C(\tau\vartheta(e_i)\otimes a^p)\in\Hom^*(\ob{P},k)$, and define $D_i^\B(a)$ to be the class of $\gamma_\B(\sigma\vartheta(e_i)\otimes a^p)\in\Hom^*(\ob{P},k)$. We need to show that $D_i^\C=D_i^\B$, and for this it suffices to prove the identity
\[
  \gamma_\B ( \sigma(w) \otimes a^p ) = \gamma_\C ( \tau(w) \otimes a^p )
\]
for all $w\in X\otimes \C(p)$ and $a\in\Hom^*(\ob{P},k)$. We can write $w = x^p\otimes f + \sum_i u_i\otimes f_i$ with $u_i\in X$ of degree $|u_i|>1-p$, and $f,f_i\in\C(p)$. By Proposition~\ref{sigmalemma} we have that
\[ \gamma_\B( \sigma(w) \otimes a^p ) = \gamma_\B( \sigma(x^p \otimes f) \otimes a^p ) = (-1)^{|f|} \gamma_\B\bigl( (\pi^{\boxtimes p}\circ f\circ \iota) \otimes a^p \bigr)  \]
On the other hand, $\gamma_\C( \tau(w)\otimes a^p ) = \gamma_\C( f\otimes a^p)$, and all these expressions equal $(-1)^{|f|(|a^p|+1)}$ times the composition
\[ \ob{P} \xrightarrow{\iota} P \xrightarrow{f} P^{\boxtimes p} \xrightarrow{\pi^{\boxtimes p}} \ob{P}^{\otimes p} \xrightarrow{a^{\otimes p}} k. \qedhere \]
\end{proof}

\begin{proof}[Proof of Theorem~\ref{producttheorem}]
Let us write $G=G_1\times G_2$. Choose complete projective resolutions $P$ and $Q$ for $k$ as trivial $kG_1$- and $kG_2$-module, respectively. Then $k\rightarrow \ob{P}\otimes\ob{Q}$ is the negative part of a projective resolution of $k$ as $kG$-module. We denote by $\B^{G_1}$, $\B^{G_2}$ and $\B^{G}$ the operads constructed above using these resolutions; then we get a quasi-isomorphism of unital operads $\B^{G_1}\otimes \B^{G_2}\rightarrow \B^G$ by tensoring morphisms. Let us denote by $A_1 = \Hom_{kG_1}^*(\ob{P},k)$, $A_2=\Hom_{kG_2}^*(\ob{Q},k)$ and $A = \Hom_{kG}^*(\ob{P}\otimes \ob{Q},k)$ the corresponding $\B$-algebras, then the commutative diagram
\[ \xymatrix{
    \B^{G_1}(p)\otimes A_1^p \otimes \B^{G_2}(p)\otimes A_2^p \ar[r]\ar[d] & A_1 \otimes A_2 \ar[d] \\
    \B^{G}(p)\otimes A^p \ar[r] & A }
\]
implies the desired result.
\end{proof}

\section{An alternative description of negative Ext-groups}\label{sectionNegativeExtGroups}
Let $n>0$. It is well-known that $\Ext_{kG}^n(A,B)=\underline{\Hom}_{kG}(\Omega^n A,B)$ admits a description via extensions of $B$ by $A$. We will now give a similar description of $\tate^{-n}_{kG}(A,B)\cong\underline{\Hom}_{kG}(A,\Omega^n B)$, which will be used throughout the next two sections.
Let us define a category $\K_n(A,B)$, whose objects are all the chain complexes
\[ C:\quad A\longrightarrow P_n\longrightarrow P_{n-1}\longrightarrow\dots \longrightarrow P_1\longrightarrow B \]
with projective modules $P_1,P_2,\dots,P_n$, and a morphism of two such complexes is a commutative diagram as follows:
\[ 
\xymatrix{
 C\ar[d] & A \ar[r]\ar@{=}[d] & P_n \ar[r]\ar[d] & P_{n-1} \ar[r]\ar[d] & \dots \ar[r] & P_1 \ar[r]\ar[d] & B\ar@{=}[d] \\
 C' & A \ar[r] & P'_{n} \ar[r] & P'_{n-1} \ar[r] & \dots \ar[r] & P'_1 \ar[r] & B  }
\]
For objects $C$ and $C'$, let us write $C\approx C'$ if there is a morphism $C\rightarrow C'$ in $\K_n(A,B)$. Define the relation $\sim$ on $\K_n(A,B)$ to be the equivalence relation generated by $\approx$, and put $K_n(A,B)=\K_n(A,B)/\sim$, the connected components of $\K_n(A,B)$. 

Let us fix a projective resolution of $B$:
\begin{align} \label{presoB}
P: && 0 \longrightarrow \Omega^n B \stackrel{i}{\longrightarrow}  P_n \longrightarrow  P_{n-1} \longrightarrow  \dots \longrightarrow P_1 \longrightarrow B \longrightarrow 0
\end{align}
\begin{thm} \label{ktheorem}
The map $\Phi:\Hom_{kG}(A,\Omega^n B)\rightarrow \K_n(A,B)$ which associates to each map $f:A\rightarrow\Omega^n B$ the complex $A\xrightarrow{\smash[b]{i\circ f}} P_n\rightarrow P_{n-1}\rightarrow\dots\rightarrow P_1\rightarrow B$ induces a bijection
$\underline{\Hom}_{kG}(A,\Omega^n B) \stackrel{1:1}{\longleftrightarrow}  K_n(A,B)$
which is natural in $G$.
\end{thm}
To prove this, we need the following proposition.
\begin{lemma} \label{liftindeplemma}
Suppose we are given two finite chain complexes $A=(0\rightarrow A_{n+1} \rightarrow \dots \rightarrow A_{0} \rightarrow 0)$ and $B=(0 \rightarrow B_{n+1} \rightarrow \dots \rightarrow  B_{0} \rightarrow 0)$, where $A_i$ is projective for $i=1,2,\dots,n$, and $B$ is exact. Let $f,g:A\rightarrow B$ be chain maps satisfying $f_{0}=g_{0}:A_{0}\rightarrow B_{0}$. Then the classes of $f_{n+1}$ and $g_{n+1}$ in $\underline{\Hom}_{kG}(A_{n+1},B_{n+1})$ are the same.
\end{lemma}
The proof is standard homological algebra, and we omit it.

\begin{proof}[Proof of Theorem \ref{ktheorem}:]
As a first step, we show that $\Phi$ induces a map $\underline{\Hom}_{kG}(A,\Omega^n B)\rightarrow K_n(A,B)$. Suppose we are given $f'\in\Hom_{kG}(A,\Omega^n B)$ such that $f'-f$ factors through some projective module $R$:
\[ \xymatrix{
  f'-f: & A \ar[r]^{u} & R \ar[r]^{w} & \Omega^n B } \]
Then the complexes $\Phi(f)$ and $\Phi(f')$ differ in their first map only; let us denote these by $\alpha, \alpha':A\rightarrow P_n$, respectively. From the commutative diagram
\[
\xymatrix{
 A \ar@{=}[d]\ar[r]^{\alpha} & P_n \ar[r] & P_{n-1} \ar[r]\ar@{=}[d] & P_{n-2} \ar[r]\ar@{=}[d] & \dots \ar[r] & B \ar@{=}[d] \\
 A \ar@{=}[d]\ar[r]^-{(\begin{smallmatrix} \alpha & u \end{smallmatrix})} & P_n\oplus R \ar[u]_{\left(\begin{smallmatrix} \id \\ 0 \end{smallmatrix}\right)} \ar[r]^-{\left(\begin{smallmatrix} \partial \\ 0 \end{smallmatrix}\right)} \ar[d]^{\left(\begin{smallmatrix} \id \\ \partial\circ w \end{smallmatrix}\right)} & P_{n-1} \ar[r]\ar@{=}[d] & P_{n-2} \ar[r]\ar@{=}[d] & \dots \ar[r] & B \ar@{=}[d] \\
 A  \ar[r]^{\alpha'} & P_n \ar[r] & P_{n-1} \ar[r] & P_{n-2}\ar[r] & \dots \ar[r] & B }
\]
we get that $\Phi(f)\sim\Phi(f')$. Therefore, we obtain a map $\underline{\Hom}_{kG}(A,\Omega^n B)\rightarrow K_n(A,B)$ which we also denote by $\Phi$. 

To construct an inverse for $\Phi$, start with some object $C=(A\rightarrow Q_*\rightarrow B)\in \K_n(A,B)$. Since the $Q_i$'s are projective and \eqref{presoB} is exact, we can lift the identity on $B$ to a map of chain complexes $f:C\rightarrow P$:
\[
\xymatrix{
  A \ar[r]\ar[d] &  Q_n \ar[r]\ar[d] &  Q_{n-1} \ar[r]\ar[d] &  \dots \ar[r] & Q_1 \ar[r]\ar[d] & B \ar@{=}[d] \\
 \Omega^n B \ar[r] &  P_n \ar[r] &  P_{n-1} \ar[r] &  \dots \ar[r] & P_1 \ar[r] & B } 
\]
By Proposition~\ref{liftindeplemma}, the stable class of the resulting map $f_{n+1}:A\rightarrow \Omega^n B$ is independent of the choice of the lift; let us write $\Psi(C)=f_{n+1}\in\underline{\Hom}_{kG}(A,B)$. Suppose we are given a morphism $g:C'\rightarrow C$ in $\K_n(A,B)$. Then $f\circ g$ is a lift of the identity on $B$ to a map of chain complexes $C'\rightarrow P$. Since $g_{n+1}=\id_A$, we have $\Psi(C')=(f\circ g)_{n+1}=f_{n+1}=\Psi(C)$. Therefore, we have constructed a map $\Psi:K_n(A,B) \rightarrow \underline{\Hom}_{kG}(A,\Omega^n B)$. The proofs of $\Psi \circ \Phi = \id$ and $\Phi\circ \Psi=\id$ are immediate.
\end{proof}

\begin{example}\label{tateelementexample}
Suppose that $p$ divides the order of the group $G$. Then it is known that $\hat{H}^{-1}(G)\cong\underline{\Hom}_{kG}(k,\Omega k)$ is isomorphic to $k$. Under the bijection of Theorem~\ref{ktheorem}, a canonical generator of that vector space is given by the complex
\[ \xymatrix@C=35pt{k \ar[r]^-{\sum_{g\in G} g} & kG \ar[r]^-{\epsilon} & k} \]
where $\epsilon$ is the augmentation of $kG$. 
\end{example}

\begin{lemma}\label{triproductlemma}
Suppose we have a commutative diagram
\[
\xymatrix{
 &  A \ar[r] \ar[d]_{g} & P_n \ar[r] \ar[d] &  \dots \ar[r] & P_1 \ar[r] \ar[d] & B \ar[d]^{f} \\
0 \ar[r] &  D \ar[r] & E_n \ar[r] & \dots \ar[r] & E_1 \ar[r] & C  \ar[r] & 0}
\]
in $\usmod kG$. Assume further that the $P_i$'s are projective, so that the upper row represents some element $\alpha\in \underline{\Hom}_{kG}(A,\Omega^n B)$, and assume that the lower row is exact, therefore representing some element $\beta\in \underline{\Hom}_{kG}(\Omega^n C,D)$. Then the diagram
\[
 \xymatrix{
   A \ar[r]^-{\alpha} \ar[d]_{g} & \Omega^n B \ar[d]^{\Omega^n f} \\
   D & \Omega^n C \ar[l]^-{\beta} } \]
commutes stably.
\end{lemma}
\begin{proof}
Choose projective resolutions $\Omega^n B \rightarrow Q_* \rightarrow B$ and $\Omega^n C\rightarrow R_*\rightarrow C$. By the usual `projective to acyclic'-argument, we get a diagram
\[ \xymatrix@C=15pt@R=15pt{
  A \ar[r] \ar[d]_{\bar{\alpha}} & P_* \ar[r] \ar[d] & B \ar@{=}[d] \\
  \Omega^n B \ar[r] \ar[d]_{\Omega^n f} & Q_* \ar[r] \ar[d] & B \ar[d]^{f} \\
  \Omega^n C \ar[r] \ar[d]_{\bar{\beta}} & R_* \ar[r] \ar[d] & C \ar@{=}[d] \\
  D \ar[r] & E_* \ar[r] & C } \]
where $\bar{\alpha}$ and $\bar{\beta}$ are unstable representatives of $\alpha$ and $\beta$, respectively. The result follows from Proposition~\ref{liftindeplemma}. 
\end{proof}

\begin{remark}\label{isomorphismremark}
Suppose we have an exact sequence $A\hookrightarrow P_n\rightarrow\dots\rightarrow P_1\twoheadrightarrow B$ with projective modules $P_1,\dots,P_n$. Then we can view this as an extension representing some stable isomorphism $\Omega^n B\rightarrow A$; but we can also consider this as an element of $\K_n(A,B)$, representing some stable isomorphism $A\rightarrow \Omega^n B$; by the previous proposition, the two maps are stable inverses of each other.
\end{remark}

We have a composition product $\K_n(B,C)\times \K_m(A,B)\rightarrow \K_{n+m}(A,C)$ similar to the Yoneda splice: given $E:A\rightarrow P_*\rightarrow B$ and $E':B\rightarrow Q_*\rightarrow C$ we define $E'\circ E$ to be the complex
\[ \xymatrix@C=10pt@R=5pt{
E'\circ E: & A \ar[r] & P_* \ar[rr]\ar[dr] & & Q_* \ar[r] & C. \\
& & & B\ar[ur] } \]
This product is compatible with the equivalence relation $\sim$ and therefore induces a product
\[ K_n(B,C) \times K_m(A,B) \rightarrow K_{n+m}(A,C). \]
\begin{lemma} 
The composition products on $K_*$ and $\tate^{-*}_{kG}$ coincide under the bijection of Theorem~\ref{ktheorem}.
\end{lemma}
\begin{proof}
Let us start with complexes $A\rightarrow P_*\rightarrow B$ and $B\rightarrow Q_*\rightarrow C$ representing stable maps $\alpha:A\rightarrow \Omega^m B$ and $\beta:B\rightarrow \Omega^n C$, respectively. Choose projective resolutions $\Omega^n C\rightarrow R_*\rightarrow C$ and $\Omega^{n+m} C\rightarrow T_*\rightarrow \Omega^n C$. Then we can lift the identity map on $C$ to commutative diagrams as follows:
\[ \xymatrix{
 A \ar[r] \ar[d]_{\bar{\gamma}} & P_*\ar[r] \ar[r] \ar[d] & B \ar[d]^{\bar{\beta}} & & B \ar[r]\ar[d]_{\bar{\beta}} & Q_* \ar[d] \ar[r] & C \ar@{=}[d] \\
 \Omega^{n+m}C \ar[r] & T_* \ar[r] & \Omega^n C & & \Omega^n C \ar[r] &  R_* \ar[r] & C  } \]
Here, $\bar{\beta}$ and $\bar{\gamma}$ are unstable representatives of $\beta$ and some $\gamma$. Note that the extension $\Omega^{n+m}C \rightarrow T_*\rightarrow \Omega^n C$ represents the identity map $\id\in \underline{\Hom}_{kG}(\Omega^m \Omega^n C,\Omega^{n+m}C)$. By Proposition~\ref{triproductlemma}, the left diagram shows that $\gamma = \beta \alpha$. After splicing the two diagrams the result follows from Proposition~\ref{liftindeplemma}.
\end{proof}
There is also a way of composing an element $x\in\tate^{-n}_{kG}(A,B)$ given as a complex $A\rightarrow P_*\rightarrow B$ with an element of $y\in\tate^m_{kG}(B,C)$ (with $m>0$) given as an extension $C\hookrightarrow M_*\twoheadrightarrow B$: 
\begin{lemma} \label{negcomposlemma}
Suppose $m<n$. The identity map of $B$ can be lifted to a diagram
\[
\xymatrix{
 A\ar[r] & P_n\ar[r] & \dots \ar[r] & P_{m+1}\ar[d]\ar[r] & P_{m}\ar[r]\ar[d] & \dots \ar[r] & P_1\ar[r]\ar[d] & B \ar@{=}[d] \\
 & & & C \ar[r] & M_m \ar[r] & \dots \ar[r] & M_1 \ar[r] & B }
\]
and for any such lifting, the complex $\xymatrix@C=15pt{ A \ar[r] & P_n\ar[r] & \dots \ar[r] & P_{m+1}\ar[r] & C}$ re\-pre\-sents the composition $y\cdot x\in\smash[t]{\tate^{m-n}_{kG}(A,C)}$.
\end{lemma}
\begin{proof}
Existence of the lifting is common homological algebra. For the second statement choose a projective resolution $\Omega^{n-m}C\rightarrow R_*\rightarrow C$; then we have the following commutative diagram:
\[ \xymatrix@R=12pt@C=12pt{
 A \ar[r] \ar[dd]^{\bar{\gamma}} & P_n \ar[r]\ar[dd] & \dots \ar[r] & P_{m+1} \ar[dr] \ar[rr]\ar[dd]  & & P_{m} \ar[r]\ar[dd] & \dots \ar[r] & P_1 \ar[r]\ar[dd] & B \ar@{=}[dd] \\
  & & & & C \ar[dr] \\
 \Omega^{n-m} C \ar[r] & R_{n-m} \ar[r] & \dots \ar[r] & R_{1} \ar[ur] \ar[rr] & & M_m \ar[r] & \dots \ar[r] & M_1 \ar[r] & B } \]
The complex in question represents the stable class of the map $\bar{\gamma}$. The bottom row represents $y\in\underline{\Hom}_{kG}(\Omega^n B,\Omega^{n-m}C)$, the upper row represents $x\in\underline{\Hom}_{kG}(A,\Omega^n B)$. The result follows from Proposition~\ref{triproductlemma}.
\end{proof}

\begin{lemma}\label{composnegtwoways}
Suppose that we have a commutative diagram
\[ \xymatrix{
   A \ar[r]\ar[d]_{f} & P_n\ar[r]\ar[d] & \dots \ar[r] & P_1\ar[r]\ar[d] &  B\ar[d]^{g} \\
   A' \ar[r] & Q_n\ar[r] & \dots \ar[r] & Q_1\ar[r] &  B'} \]
with projective modules $P_i,Q_i$ for $i=1,2,\dots,n$. Then the rows represent maps $x:A\rightarrow \Omega^n B$ and $y:A'\rightarrow \Omega^n B'$, respectively, and $y\circ f = \Omega^n(g) \circ x$ in $\stmod kG$.
\end{lemma}
\begin{proof}
Choose a projective resolution $\Omega^n B'\rightarrow R_*\rightarrow B'$. By usual homological algebra, we get a diagram
\[
\xymatrix@R=15pt{
 A \ar[r] \ar[d]_{f'} & P_* \ar[r] \ar[d] & B \ar[d]^{g} \\
 A' \ar[r] \ar[d]_{y} & Q_* \ar[r] \ar[d] & B' \ar@{=}[d] \\
 \Omega^n B' \ar[r] & R_* \ar[r] & B'
} \]
and then the result follows from Proposition~\ref{triproductlemma}.
\end{proof}

\begin{remark} \label{compospostwoways}
There is a similar statement for extensions. Suppose that we have a diagram as in Proposition~\ref{composnegtwoways}, but this time with exact rows and the $P_i$'s and $Q_i$'s are not necessarily projective. Then the rows represent maps $x:\Omega^n B\rightarrow A$ and $y:\Omega^n B'\rightarrow A'$, and $f\circ x = y\circ \Omega^n(g)$ in $\stmod kG$.
\end{remark}

\begin{remark} \label{complexcohomologyclass} 
If $k\rightarrow I_0\rightarrow I_1\rightarrow\dots\rightarrow I_j$ is an injective resolution and $\lambda:I_j\rightarrow k$ represents some cohomology class in $\hat{H}^{-1-j}(G)\cong H^{j} \Hom(I,k)$, then the complex $k\rightarrow I_0\rightarrow\dots\rightarrow I_j\xrightarrow{\lambda} k$ represents the same class.
\end{remark}
 

\section{The Evens norm map and the dual operations}\label{sectionEvensNorm}
In this part we are going to show that some of the dual operations $Q_i^*$ on ordinary group cohomology are compatible with the Evens norm map in certain cases. For simplicity we restrict to the case $p=2$. Recall (see e.g.~\cite{benson2}, \S~4.1) that the Evens norm map is a function
\[ \norm_{K,G}: H^i(K)\rightarrow H^{ni}(G) \]
for all $i\geq 0$, where $G$ is a finite group and $K\leq G$ is a subgroup of index $n$. It can be defined as follows: let $x\in H^i(K)=\Ext^i_{kK}(k,k)$ be represented by an exact sequence $k=E_i\rightarrow E_{i-1}\rightarrow\dots\rightarrow E_0\rightarrow k$, which we think of as an augmented complex $E\rightarrow k$. Then define $x^{\otimes n}$ to be the augmented complex $E^{\otimes n}\rightarrow k$, which is an exact sequence of $k(\Sigma_n\wr K)$-modules. It therefore represents some class in $H^{ni}(\Sigma_n\wr K)$. We then choose a suitable inclusion $\iota:G\hookrightarrow \Sigma_n\wr K$ and define $\norm_{K,G}(x) = \iota^*(x^{\otimes n})$.

In the following, we will often implicitly identify $H^*(G)$ with the dual of $\hat{H}^{-1-*}(G)$ by the use of Tate duality. In particular, we have dual operations $Q_i^*:H^{i+j}(G)\rightarrow H^j(G)$.
\begin{thm} \label{dualQandTensor} Let $k=\Ff_2$, and let $K$ be a subgroup of index $n$ of a finite group $G$. 
\begin{enumerate} 
 \item \label{dualQandTensorA} For all $i\geq 0$ the diagram
        \[ \xymatrix{ H^i(K) \ar[r]^{Q_i^*} \ar[d]_{(-)^{\otimes n}} &  H^0(K) \ar[d]^{(-)^{\otimes n}} \\
                      H^{ni}(\Sigma_n\wr K) \ar[r]^{Q_{ni}^*} & H^{0}(\Sigma_n\wr K) } \]
       commutes.
 \item \label{dualQandTensorB} If $K$ is a central factor of $G$ (e.g.,~a central subgroup or a direct factor), then for $x\in H^i(K)$ we have that $\norm_{K,G} Q_i^*(x) = Q_{ni}^*(\norm_{K,G} x)$.
 \item \label{dualQandTensorC} For $x\in H^i(K)$ we have that $Q_i^*(x) = Q_{ni}^*(x^n) \in H^0(K)$.
\end{enumerate}
\end{thm}
\begin{remark}
Recall that $K$ is a central factor of $G$ if and only if the product of $K$ with its centralizer is the whole group $G$. The condition we really need for the proof of part \eqref{dualQandTensorB} is that we can choose coset representatives for $K$ in $G$ which commute with all elements of order $2$ in $K$. This is true if $K$ is a central factor, but it is also true in other cases like $\Zz/4\Zz\subseteq Q_8$. The condition is not satisfied for $\Zz/2\Zz\times \Zz/2\Zz\subset D_8$ (the dihedral group with $8$ elements), and we will see in Remark~\ref{normNotGeneral} that the conclusion fails in that case.  
\end{remark}
\begin{cor}
If the order of the finite group $G$ equals an odd multiple of $2^i$ with $i\geq 1$, then the operation $Q_n:\hat{H}^{-1}(G)\rightarrow \hat{H}^{-1-n}(G)$ is non-trivial whenever $n$ is divisible by $2^i$.
\end{cor}
\begin{proof}
Let $P\leq G$ be a $2$-Sylow subgroup, which is of order $2^i$. The commutative diagram
\[ \xymatrix{ \hat{H}^{-1}(G) \ar[d]^{\res_{P,G}}_{\cong} \ar[r]^{Q_n} & \hat{H}^{-1-n}(G) \ar[d]^{\res_{P,G}} \\ 
              \hat{H}^{-1}(P) \ar[r]_{Q_n} & \hat{H}^{-1-n}(P) } 
\]
shows that it is enough to consider the case of a $2$-group $P$. Let $K\leq P$ be a central cyclic subgroup of order $2$; then the commutative diagram
\[ \xymatrix{ H^{2^ij}(P) \ar[r]^{Q_{2^ij}^*} & H^0(P) \\ 
              H^j(K) \ar[u]^{\norm} \ar[r]_{Q_j^*}^{\cong} & H^0(K) \ar[u]^{\norm}_{\cong} }
\]
proves the claim.
\end{proof}

For the proof of Theorem~\ref{dualQandTensor} we use the following reinterpretation of $Q_i:\hat{H}^{-1}(G)\rightarrow \hat{H}^{-1-i}(G)$.
\begin{lemma}\label{interpretqi}
Let $G$ be a finite group of order divisible by $p=2$, and let $\varphi\in \hat{H}^{-1}(G)$ be the canonical generator.
For every $i\geq 0$, the complex 
\[ k \xrightarrow{N^{\otimes 2}}  kG^{\otimes 2} \xrightarrow{1+T}  kG^{\otimes 2} \xrightarrow{1+T}  \dots \xrightarrow{1+T} kG^{\otimes 2} \xrightarrow{\varepsilon^{\otimes 2}} k \]
with $i+1$ projective modules $kG^{\otimes 2}$ represents the element $Q_i(\varphi)\in \hat{H}^{-1-i}(G)$. Here, $T$ denotes the twist map interchanging the two factors of $kG^{\otimes 2}$, and $\varepsilon$ is the augmentation map.
\end{lemma}
\begin{proof}
We use the definition of $Q_i$ using the operad $\B$ from \S\ref{productsofgroups}. Let $P$ be a complete projective resolution of $k$ as trivial $kG$-module, where we assume that $P_{-1}=kG$ and $k\hookrightarrow P_{-1}$ is the norm map $N$. Let $W$ be the standard free resolution of $k$ as trivial $k\Sigma_2$-module, with one generator $e_j$ in degree $j$ for every $j\geq 0$. Choose a $\Sigma_2$-equivariant chain map $\Psi:W\rightarrow \B(2)$ lifting the identity of $k$, and define $\alpha_j:P_{-1-j}\rightarrow P_{-1}^{\otimes 2}$ to be the degree $0$-part of the map $\Psi(e_j)\in\B(2)\subseteq \Hom^*(\overline{P},\overline{P}^{\otimes 2})$. We get a commutative diagram
\[
\xymatrix{ k \ar[r]^-{N} \ar@{=}[d] & P_{-1} \ar[d]^{\alpha_0} \ar[r] & P_{-2} \ar[d]^{\alpha_1} \ar[r] & P_{-3} \ar[d]^{\alpha_2} \ar[r] & \dots \ar[r] & P_{-i} \ar[d]^{\alpha_{i-1}} \ar[r]^{\lambda} & k \ar@{=}[d] \\
           k \ar[r]_-{N^{\otimes 2}} & P_{-1}^{\otimes 2} \ar[r]_{1+T} & P_{-1}^{\otimes 2} \ar[r]_{1+T} & P_{-1}^{\otimes 2} \ar[r]_{1+T} & 
             \dots \ar[r]_{1+T} & P_{-1}^{\otimes 2} \ar[r]_{\varepsilon^{\otimes 2}} & k  } \]
By definition, the class of $\lambda$ in $H^i \Hom(\overline{P},k)\cong \hat{H}^{-1-i}(G)$ represents $Q_i(\varphi)$. The commutative diagram then shows the claim by Remark~\ref{complexcohomologyclass}. 
\end{proof}

\begin{lemma} \label{t2representsinHminusOne}
Let $G$ be a finite group of order divisible by $p=2$, and let $\nu:k\rightarrow kG$ be the norm map. If $\alpha:kG^{\otimes 2}\rightarrow k$ is a map for which $k\xrightarrow{\nu^{\otimes 2}} kG^{\otimes 2} \xrightarrow{\alpha} k$ is a complex, then that complex represents $\sum_{g\in G} \alpha(1\otimes g)\in k=\hat{H}^{-1}(G)$. 
\end{lemma}
\begin{proof}
Let $b=\sum_g\alpha(1\otimes g)$; then the result follows from the commutative diagram
\[
 \xymatrix{
   k \ar[r]^{\nu^{\otimes 2}} & kG^{\otimes 2} \ar[r]^{\alpha} & k \\
   k \ar[r]^{\nu} \ar@{=}[u] \ar[d]^{b} & kG \ar[r]^{b\varepsilon} \ar[u]^{1\otimes \nu} \ar[d]^{b} & k \ar@{=}[u] \ar@{=}[d] \\
   k \ar[r]^{\nu} & kG \ar[r]^{\varepsilon} & k }   
\]
and Example~\ref{tateelementexample} and Proposition~\ref{composnegtwoways}.
\end{proof}

\begin{lemma}\label{technicalSigmanLemma}
 Suppose that $\sigma\in\Sigma_n$ satisfies $\sigma^2=1$ and $\sigma\neq 1$, and let $K$ be a finite group. Define the map $f:\Zz/2\Zz\times K\rightarrow \Sigma_n\wr K$ to be $(u,g)\mapsto (\sigma^u;g,g,\dots,g)$. Then there is some $m\geq 1$ such that for every $x\in H^i(K)$ we have $f^*(x^{\otimes n})=x^{n-2m}\bigl(\sum_{r=0}^{|x|} \Sq^r(x)z^{|x|-r}\bigr)^m$, where $H^*(\Zz/2\Zz)=k[z]$.
\end{lemma}
\begin{proof}
We can assume that $\sigma$ is of the form $(1\,\, 2)\, (3\,\, 4)\, \dots (2m-1\,\, 2m)$ for some $m$. For every $i$ and every group $L$ denote by $\Psi_{i,L}$ the map $L\rightarrow \Sigma_i \wr L$ given by $l\mapsto (\id;l,l,\dots,l)$. Let $h$ be the composition
\[
 \Zz/2\Zz \times K \rightarrow \Sigma_2\wr K \xrightarrow{\Psi_{m,\Sigma_2\wr K}} \Sigma_m\wr(\Sigma_2\wr K) \hookrightarrow \Sigma_{2m}\wr K,
\]
the first map being given by $(u;g) \mapsto (\tau^u;g;g)$, where $\tau$ is the generator of $\Sigma_2$. Also let $j$ be the composition $\Zz/2\Zz\times K \xrightarrow{\text{proj}} K \xrightarrow{\Psi_{n-2m,K}} \Sigma_{n-2m}\wr K$; then we get a composition
\[
 \Zz/2\Zz \times K \xrightarrow{h\times j} \Sigma_{2m}\wr K \times \Sigma_{n-2m}\wr K \hookrightarrow \Sigma_n \wr K
\]
which equals $f$. Now let $x\in H^i(K)$; then $x^{\otimes n}\in H^{ni}(\Sigma_n\wr K)$ restricts to $(x^{\otimes 2m})\otimes(x^{\otimes (n-2m)})\in H^{ni}(\Sigma_{2m}\wr K\times \Sigma_{n-2m}\wr K)$. Now
\[ h^*(x^{\otimes 2m}) = \res \Psi^*_{m,\Sigma_2\wr K}( (x^{\otimes 2})^{\otimes m}) = (\res x^{\otimes 2})^m = \bigl(\norm_{K,\Zz/2\Zz\times K}(x)\bigr)^m. \]
On the other hand, $j^*(x^{\otimes (n-2m)}) = x^{n-2m}$, so it remains to show that
\[ \norm_{K,\Zz/2\Zz\times K}(x) = \sum_{r=0}^{|x|} \Sq^r(x)z^{|x|-r}, \]
which is done, e.g., in ~\cite{benson2}, \S4.4; note that this can actually be used to define the Steenrod operations on ordinary group cohomology. 
\end{proof}

Consider the augmentation $k\Sigma_n\rightarrow k$ as an augmented chain complex; then the augmented chain complex $k\Sigma_n\otimes k\Sigma_n\rightarrow k$ is a chain complex of right $\Sigma_2\times \Sigma_n$-modules, where $\Sigma_n$ acts diagonally and $\Sigma_2$ acts by permuting the factors. Let $W$ be the standard free resolution of $k$ as trivial $k\Sigma_2$-module. By endowing $W$ with a trivial right $\Sigma_n$-action, we can consider $W\rightarrow k$ and hence also $W\otimes k\Sigma_n\otimes k\Sigma_n\rightarrow k$ as augmented chain complexes of right $\Sigma_2\times \Sigma_n$-modules. As such the latter consists entirely of free modules, and we can therefore lift the identity map of $k$ to a map of chain complexes $\vartheta:W\otimes k\Sigma_n\otimes k\Sigma_n\rightarrow W^{\otimes n}$, where $\Sigma_n$ acts on $W^{\otimes n}$ by permuting the factors and $\Sigma_2$ acts diagonally. Finally note that $k\Sigma_n$ is a right $\Sigma_n\wr K$-module via the projection map $\Sigma_n\wr K\rightarrow \Sigma_n$, and $kK^{\otimes n}$ is also a right $\Sigma_n\wr K$-module. Therefore the tensor product $k\Sigma_n\otimes kK^{\otimes n}$ is a right $k(\Sigma_n\wr K)$-module, which is free of rank one. We can now form the following map $\xi$ of augmented chain complexes over $k(\Sigma_n\wr K)$:
\[
\xymatrix{
 W\otimes_{\Sigma_2} (k(\Sigma_n\wr K))^{\otimes 2} \ar[r]^-{\cong} \ar[dr]_{\xi} & W\otimes_{\Sigma_2} (k\Sigma_n \otimes kK^{\otimes n})^{\otimes 2} \ar[r]^-{\cong} &
  (W\otimes k\Sigma_n\otimes k\Sigma_n) \otimes_{\Sigma_2} (kK^{\otimes n})^{\otimes 2} \ar[d]^{\vartheta\otimes \text{twist}} \\
 & (W\otimes_{\Sigma_2} kK^{\otimes 2})^{\otimes n} & W^{\otimes n} \otimes_{\Sigma_2} (kK^{\otimes 2})^{\otimes n} \ar[l]^{\text{twist}} }
\]
In the following, we consider triples $(E,\bar{\beta},\bar{\gamma})$ where $E\rightarrow k$ is an exact sequence $k=E_i \hookrightarrow E_{i-1} \rightarrow \dots \rightarrow E_0 \twoheadrightarrow k$ of $kK$-modules, $\bar{\beta}:W\otimes_{\Sigma_2} kK^{\otimes 2} \rightarrow E$ is a map of augmented chain complexes, and $\bar{\gamma}$ is defined to be $\bar{\gamma} = \bar{\beta}^{\otimes n} \circ \xi$. Then $\bar{\gamma}:W\otimes_{\Sigma_2} (k(\Sigma_n\wr K))^{\otimes 2}\rightarrow E^{\otimes n}$ is a map of augmented chain complexes over $k(\Sigma_n \wr K)$. Define $\beta$ to be the composite
\[ kK^{\otimes 2} \cong W_i\otimes_{\Sigma_2} kK^{\otimes 2} \xrightarrow{\bar{\beta}_i} E_i=k, \]
and similarly define $\gamma:k(\Sigma_n\wr K)^{\otimes 2}\rightarrow k$. Furthermore, for every group $L$ let us define the subset 
$L' = \{ l\in L \,\mid\, l^2 = 1 \} \subseteq L$.
\begin{lemma} \label{LprimeIsEnough}
 If $(E,\bar{\beta},\bar{\gamma})$ is a triple as above, then $\sum_{g\in K} \beta(1\otimes g) = \sum_{g\in K'} \beta(1\otimes g)$ and $\sum_{l\in L} \gamma(1\otimes l) = \sum_{l\in L'} \gamma(1\otimes l)$.
\end{lemma}
\begin{proof}
The formula
\begin{align}  \label{betaformula}
  \bar{\beta}_i(w\otimes_{\Sigma_2} (g\otimes h)) = \varepsilon(w)\beta(g\otimes h), 
\end{align}
holds because it is true for $w=1\in\Sigma_2$ and for $w=1-\tau\in k\Sigma_2$ (where $\tau$ is the generator of $\Sigma_2$) since $\bar{\beta}$ is a chain map. The formula implies that $\beta(1\otimes g) = \beta(g\otimes 1) = \beta(1\otimes g^{-1})$, and therefore $\sum_{g\in K} \beta(1\otimes g) = \sum_{g\in K'} \beta(1\otimes g)$. The same proof applies to $\gamma$.
\end{proof}

\begin{lemma}
 There exist constants $c_{n,i,\sigma}$ (for all $\sigma\in\Sigma_n$), not depending on $K$, with the following property: for all triples $(E,\bar{\beta},\bar{\gamma})$ as above and all elements $g=(\sigma,k_1,k_2,\dots,k_n)\in \Sigma_n\wr K$ we have that
\[ \gamma(1\otimes g) = c_{n,i,\sigma}\cdot \beta(1\otimes k_1)\cdots \beta(1\otimes k_n). \]
\end{lemma}
Some of the constants will be determined later in Proposition~\ref{constantsProp}.
\begin{proof}
We can write $\vartheta_{ni}(1\otimes 1\otimes \sigma) = \sum_s w_{s,1}\otimes\dots\otimes w_{s,n}$ where $w_{s,j}\in W$. Then $\gamma(1\otimes g)$ equals 
\[ \gamma(1\otimes g) = \sum_s \bar{\beta}^{\otimes n} \Bigl( \bigl(w_{s,1}\otimes_{\Sigma_2} (1\otimes k_1) \bigr) \otimes \dots \otimes \bigl(w_{s,n}\otimes_{\Sigma_2} (1\otimes k_n)\bigr) \Bigr). \]
If the degree of one of the $w_{s,j}$ is bigger than $i$, then the corresponding $s$-th summand vanishes because $\bar{\beta}$ is the zero map. Therefore, we are only interested in the case where all $w_{s,j}$ are of degree~$i$, in which case we can simplify by \eqref{betaformula}
\begin{align*}
 \gamma(1\otimes g)  
  &= \Bigl( \sum_s \varepsilon(w_{s,1})\cdots\varepsilon(w_{s,n}) \Bigr) \cdot \beta(1\otimes k_1)\beta(1\otimes k_2)\cdots \beta(1\otimes k_n) \\
  &= c_{n,i,\sigma} \cdot \beta(1\otimes k_1)\beta(1\otimes k_2)\cdots \beta(1\otimes k_n),
\end{align*}
where $c_{n,i,\sigma}$ is some constant in $k$ not depending on the group $K$. 
\end{proof}

\begin{lemma}
Let $L$ be any finite group, and suppose that $f:L\rightarrow \Sigma_n\wr K$ is an injective group homomorphism and $c\in k$ is some constant. Suppose that for all triples $(E,\bar{\beta},\bar{\gamma})$ as above we have that $c\cdot \sum_{g\in K} \beta(1\otimes g) = \sum_{l\in L} \gamma(1\otimes f(l))$. Then for all $x\in H^i(K)$ the formula $Q_{ni}^* f^*(x^{\otimes n}) = c\cdot f^*(Q_i^*(x)^{\otimes n})\in H^0(L)$ holds.
\end{lemma}
\begin{proof}
We may assume that the order of both $K$ and $L$ is divisible by $p=2$. Then we identify $\hat{H}^0(K)$ and $\hat{H}^0(L)$ with $k$, so that we have to prove
\[ Q_{ni}^* f^*(a^{\otimes n}) = c \cdot Q_i^*(a) \]
for all $a\in H^i(K)$ (recall that $k=\Ff_2$). 
Let $E$ be an exact sequence representing $a\in H^i(K)$. Since the modules of the augmented complex $W\otimes kK^{\otimes 2}\rightarrow k$ are free, we can lift the identity of $k$ to a chain map $\bar{\beta}$:
 \[ \xymatrix{  k \ar[r]^-{\nu^2} & kK^{\otimes 2} \ar[r]^{1+T} \ar[d]_{\beta} & kK^{\otimes 2} \ar[r]^{1+T} \ar[d] & \dots \ar[r]^{1+T} & kK^{\otimes 2}  \ar[r]^-{\varepsilon^2} \ar[d] & k \ar@{=}[d] \\
    & k \ar[r] & E_{i-1} \ar[r] & \dots \ar[r] & E_0 \ar[r] & k } \]
Here $\nu:k\rightarrow kK$ is the norm map, and the upper row represents $Q_i(\kappa)\in \hat{H}^{-1-i}(K)$ for the generator $\kappa\in \hat{H}^{-1}(K)$ (by Proposition~\ref{interpretqi}). Due to Proposition~\ref{negcomposlemma} the product $Q_i(\kappa)a\in \hat{H}^{-1}(K)$ is represented by the complex $k \xrightarrow{\nu^2} kK^{\otimes 2} \xrightarrow{\beta} k$. Therefore, by Proposition~\ref{t2representsinHminusOne}, $Q_i(\kappa) a = \sum_{g\in K} \beta(1\otimes g) \kappa$ and hence   
\begin{align} \label{fcompareQdualLemma1} 
 Q_i^*(a) = \sum_{g\in K} \beta(1\otimes g) \in k.
\end{align}
As before, we get a triple $(E,\bar{\beta},\bar{\gamma})$ in such a way that the diagram of $kL$-modules
\[ \xymatrix{
      k \ar[r]^-{\mu^2} & kL^{\otimes 2} \ar[r]^{1+T} \ar[d]_{(kf)^{\otimes 2}} & kL^{\otimes 2} \ar[r]^{1+T} \ar[d] & \dots \ar[r]^{1+T} & kL^{\otimes 2}  \ar[r]^-{\varepsilon^2} \ar[d] & k \ar@{=}[d] \\
                        & k(\Sigma_n\wr K)^{\otimes 2} \ar[r]^{1+T} \ar[d]_{\gamma} & k(\Sigma_n\wr K)^{\otimes 2} \ar[r]^-{1+T} \ar[d] & \dots \ar[r]^-{1+T} & k(\Sigma_n\wr K)^{\otimes 2}  \ar[r]^-{\varepsilon^2} \ar[d] & k \ar@{=}[d] \\
                        & k \ar[r] & (E^{\otimes n})_{ni-1} \ar[r] & \dots \ar[r] & (E^{\otimes n})_0 \ar[r] & k }
\]
commutes, where $\mu:k\rightarrow KL$ is the norm map, such that the upper row represents $Q_{ni}(\lambda)\in \hat{H}^{-1-ni}(L)$, where $\lambda\in\hat{H}^{-1}(L)$ is the generator. As above, Propositions~\ref{negcomposlemma} and~\ref{t2representsinHminusOne} show that
$Q_{ni}(\lambda) f^*(a^{\otimes n}) = \sum_{l\in L} \gamma(1\otimes f(l)) \lambda$, so that
\begin{align} \label{fcompareQdualLemma2}
 Q_{ni}^*(f^*(a^{\otimes n})) = \sum_{l\in L} \gamma(1\otimes f(l)) \in k.
\end{align}
Combining formulas \eqref{fcompareQdualLemma1} and \eqref{fcompareQdualLemma2} we get the desired result.
\end{proof}
We will now exploit this fact for several maps $f$.
\begin{lemma}\label{constantsProp}
The constants $c_{n,i,\sigma}$ satisfy 
\begin{align}
   c_{n,i,\id} &= 1,  \label{constantsClaimA} \\
   c_{n,i,\sigma} &= 0 &&\text{if $\sigma^2=1$ and $\sigma\neq \id$, and}  \label{constantsClaimB} \\
   c_{n,i,\sigma} &= c_{n,i,\sigma^{-1}} &&\text{for all $\sigma$.}  \label{constantsClaimC} 
\end{align}
\end{lemma}
\begin{proof}[Proof of Proposition~\ref{constantsProp} and Theorem~\ref{dualQandTensor}]
As a first step, take $L=K$ and let $f:K\rightarrow \Sigma_n\wr K$ be given by $g\mapsto (\id;g,\dots,g)$ for all $g\in K$. Then $x^n = f^*(x^{\otimes n})$, and the computation
\[ \sum_{g\in K} \gamma(1\otimes f(g)) = \sum_{g\in K} \gamma(1\otimes(\id;g,\dots,g)) = c_{n,i,\id} \sum_{g\in K} \beta(1\otimes g)^n = c_{n,i,\id}\sum_{g\in K} \beta(1\otimes g) \]
shows that $Q_i^*(x) = c_{n,i,\id} Q_{ni}^*(x^n)\in H^0(K)$. If we put $K=\Zz/2\Zz$, then the computations in Example~\ref{examplecyclic2} show that the constant $c_{n,i,\id}$ equals $1$, so we have proved \eqref{constantsClaimA} and Theorem~\ref{dualQandTensor}.\eqref{dualQandTensorC}.

As a second step, let us take $L=\Zz/2\times K$ and let $f:\Zz/2\Zz\times K\rightarrow \Sigma_n\wr K$ to be given by $(u,g)\mapsto (\sigma^u;g,g,\dots,g)$, where $\sigma\in\Sigma_n$ is some fixed element of order $2$. Then 
\begin{align*}
 \sum_{g\in \Zz/2\Zz\times K} \gamma(1\otimes f(g)) &= \sum_{g\in K} \gamma(1\otimes(\id;g,\dots,g)) + \gamma(1\otimes(\sigma;g,\dots,g))  \\
           &= (1+c_{n,i,\sigma}) \sum_{g\in K} \beta(1\otimes g).
\end{align*}
We take $K=\Zz/2\Zz$, but we keep the notation $K$ in order to distinguish from the other factor $\Zz/2\Zz$. We have $H^*(K)\cong k[x]$ and $H^*(\Zz/2\Zz)\cong k[z]$ for one-dimensional classes $x$ and $z$. By Proposition~\ref{technicalSigmanLemma}, we know that $f^*((x^i)^{\otimes n}) = f^*(x^{\otimes n})^i = x^{(n-2m)i}(x^2+xz)^{mi}$. By the computations in Example~\ref{exampleC2C2}, applying $Q_{ni}^*$ to such a polynomial in $x,z$ equals the sum of the evaluations of that polynomial at $(x,z)=(1,1)$, $(0,1)$ and $(1,0)$; therefore, $Q_{ni}^*(f^*((x^i)^{\otimes n})) = 1\in H^0(\Zz/2\Zz\times K)$. This implies \eqref{constantsClaimB}.

In order to prove \eqref{constantsClaimC}, take a situation in which $l\in K$ is of order $2$, and $\beta$ is such that $\beta(1\otimes l)\neq 0$. Then put $g=(\sigma;l,l,\dots,l)$ and the result follows from $\gamma(1\otimes g)=\gamma(1\otimes g^{-1})$. Up to this point, we have proved Proposition~\ref{constantsProp} completely. 

Now we prove Theorem~\ref{dualQandTensor}.\eqref{dualQandTensorA}. Take $f:L\rightarrow \Sigma_n\wr K$ to be the identity map of $\Sigma_n\wr K$ and compute
\begin{align*}
 \sum_{g=(\sigma;k_1,\dots,k_n)\in \Sigma_n\wr K} \gamma(1\otimes g) &= \sum_{\sigma} c_{n,i,\sigma} \cdot \sum_{k_1,\dots,k_n\in K} \beta(1\otimes k_1)\beta(1\otimes k_2)\cdots \beta(1\otimes k_n) \\
&= \sum_{\sigma} c_{n,i,\sigma} \bigl( \sum_{g\in K} \beta(1\otimes g) \bigr)^n.
\end{align*}
By Proposition~\ref{constantsProp}, $\sum_{\sigma\in\Sigma_n} c_{n,i,\sigma} = 1$, which proves Theorem~\ref{dualQandTensor}.\eqref{dualQandTensorA}.

Let $\pi:\Sigma_n\wr K\rightarrow \Sigma_n$ be the projection map. For the proof of Theorem~\ref{dualQandTensor}.\eqref{dualQandTensorB}, choose a set $\{g_i\}$ of coset representatives of $K$ in $G$ with the property that all the $g_i$'s commute with all elements of order $2$ in $K$. Then for each $g\in G$, there are unique elements $k_1,\dots,k_n\in K$ and $\sigma\in\Sigma_n$ such that $g g_j = g_{\sigma(j)} k_j$ for all $j$, and we get an injection $f:G\hookrightarrow \Sigma_n\wr K$ by $g\mapsto (\sigma;k_1,\dots,k_n)$. Then $\norm_{K,G} (x) = f^*(x^{\otimes n})$, and we need to investigate  
\begin{align*}
 \sum_{g\in G} \gamma(1\otimes g) &= \sum_{g\in G'} \gamma(1\otimes g)  &&\text{by Proposition~\ref{LprimeIsEnough}} \\
&= \sum_{\substack{g\in G' \\ \pi(f(g))=\id}} \gamma(1\otimes g) &&\text{by \eqref{constantsClaimA} and \eqref{constantsClaimB}.}
\end{align*}

But if $\pi(f(g))=\id$, then $g g_j = g_j k_j$ for all $j$, which means $gg_j = k_j g_j$ for all $j$ by our condition on $K$. Therefore, we get $g=k_j$ for all $j$, so that $g\in K$ and $f(g)=(\id;g,g,\dots,g)$. Conversely, if $g\in K$, then $f(g)=(\id;g,g,\dots,g)$ by our condition on $K$. Therefore,
\begin{align*}
  \sum_{g\in G} \gamma(1\otimes g) &= \sum_{g\in K'} \gamma(1\otimes (\id;g,g,\dots,g)) = \sum_{g\in K'} \beta(1\otimes g)^n = \sum_{g\in K'} \beta(1\otimes g).
\end{align*}
This proves \eqref{dualQandTensorB} of Theorem~\ref{dualQandTensor}.
\end{proof}

\begin{example} Let us work out in detail the operations $Q$ on the generator of $\hat{H}^{-1}(G)$ in the case $G=D_8$, the dihedral group with $8$ elements. The structure of the cohomology ring $H^*(G)$ is known to be $H^*(G) \cong k[a,b,c] / (ab)$ where $|a|=|b|=1$ and $|c|=2$ (see, e.g., \cite{Carlson}, Theorem~7.8). From \cite{BCProductsNegativeCohomology}, Theorem~3.1 and Lemma~2.1 we get that $\hat{H}^-(G)\cdot \hat{H}^-(G)=0$ and that $\hat{H}^n(G)\cdot \hat{H}^m(G)=0$ for $n<0\leq n+m$. Consider the $k$-basis $\{a^ic^j,b^ic^j\}_{i,j\geq 0}$ of $H^*(G)$ and let us define $\{\varphi_{a^ic^j},\varphi_{b^ic^j}\}$ to be the dual basis; in particular, $\varphi_1$ is the canonical generator of $\hat{H}^{-1}(G)$. Using Tate duality, one derives the relations
\begin{align*}
 c \varphi_{a^ic^j} &= \begin{cases} \varphi_{a^ic^{j-1}} & \text{if $j>0$,} \\ 0 & \text{otherwise,} \end{cases} \\
 a \varphi_{a^ic^j} &= \begin{cases} \varphi_{a^{i-1}c^j} & \text{if $i>0$,} \\ 0 & \text{otherwise,} \end{cases} \\
 a \varphi_{b^ic^j} &= 0, 
\end{align*}
and similarly for $a$ and $b$ interchanged. All these facts together completely determine the multiplicative structure of $\hat{H}^*(G)$.

From $Q_1(\varphi_1)=\varphi_1^2=0$ we get that $Q_1^*:H^1(G)\rightarrow H^0(G)$ is zero. Therefore, by Theorem~\ref{dualQandTensor}.\eqref{dualQandTensorC}, we get $Q_2^*(a^2)=Q_2^*(b^2)=0$. Now notice that $G\cong \Sigma_2\wr \Zz/2\Zz$, so that Theorem~\ref{dualQandTensor} implies that $Q_2^*:H^2(G)\rightarrow H^0(G)$ is onto and hence $Q_2^*(c)=1$. We have therefore determined $Q_2(\varphi_1)=\varphi_c$. More generally, note that $Q(\varphi_1)=Q(a\varphi_a)=Q(a)Q(\varphi_a)=(a+a^2)Q(\varphi_a)$ is divisible by $a$ and, by symmetry, also by $b$. This fact already implies that $Q_i(\varphi_1)$ is a multiple of $\varphi_{c^j}$ for some $j$. Together with $Q_{2i}^*(c^i) = Q_2^*(c)^i=1$ we get 
\begin{align*}
 Q_{2i}(\varphi_1) &= \varphi_{c^i}   && \text{for $i\geq 1$,} \\
 Q_{2i+1}(\varphi_1) &= 0  && \text{for $i\geq 0$.}  
\end{align*}
\end{example}

\begin{remark} \label{normNotGeneral}
Let us prove that Theorem~\ref{dualQandTensor}.\eqref{dualQandTensorB} is not true for arbitrary subgroups $K$ of $G$. Take $K=\Zz/2\Zz\times \Zz/2\Zz$ and $G=D_8$, and let us write $\norm$ for $\norm_{K,G}$. We know that $H^*(K)\cong k[x,y]$ for some one-dimensional classes $x,y$. Suppose that Theorem~\ref{dualQandTensor}.\eqref{dualQandTensorB} would hold in that case. Then $Q_2^*(\norm(x))=\norm(Q_1^*(x))=0$, so that $\norm(x)=\alpha a^2 + \beta b^2$ for some $\alpha,\beta\in k$. Similarly, $\norm(y)=\alpha' a^2+\beta' b^2$ for some $\alpha',\beta'\in k$. But then $\norm(xy)=\norm(x)\norm(y)=\alpha\alpha' a^4 + \beta\beta' b^4$ and hence
\[ 0 = Q_4^*(\norm(xy)) \neq \norm(Q_2^*(xy)) = \norm(1), \]
a contradiction. 
\end{remark}

\section{Productive elements at the prime 2}
Let $G$ be a finite group, and let $\zeta:\Omega^n k\rightarrow k$ be a surjective map representing a Tate cohomology class $[\zeta]\in\hat{H}^n(G)$. Define $L_\zeta$ to be the kernel of $\zeta$; we therefore get an exact triangle 
\begin{align} \label{zetaTriangle}
 \Omega k \xrightarrow{\eta} L_\zeta \xrightarrow{\iota} \Omega^n k \xrightarrow{\zeta} k.
\end{align}
Following Carlson (\cite{Carlson}, \S9) we call the class $[\zeta]$ \emph{productive} if $\zeta$ annihilates the cohomology of $L_\zeta$, that is, the map $\zeta\otimes\id_{L_\zeta}:\Omega^n k\otimes L_\zeta\rightarrow L_\zeta$ is stably zero. It is known that, for all primes $p\geq 3$, a non-zero class $[\zeta]$ is productive if and only if its degree $n$ is even (see \cite{CarlsonPapr}, Theorem~4.1). The case $p=2$ is more complicated, and we will show in this section that the operations $Q$ can be used to determine whether a class is productive or not:
\begin{thm}\label{productiveTheorem}
Let $p=2$, and let $G$ be a finite group. A cohomology class $[\zeta]\in\hat{H}^n(G)$ is productive if and only if $\Pp_1(\zeta)$ is divisible by $\zeta$ in $\hat{H}^*(G)$. 
\end{thm}
\begin{remark}
The 'only if' part of this theorem has been conjectured and independently proven in the case of ordinary cohomology classes by Erg\"{u}n Yal\c{c}in, using connections to the existence of diagonal approximations of certain chain complexes.
\end{remark}

The proof of Theorem~\ref{productiveTheorem} relies on the following commutative diagram.
\begin{lemma}\label{commutativeSquare}
Under the conditions of the theorem, the following diagram commutes stably:
\[ \xymatrix{
      \Omega^n k \otimes L_\zeta \ar[r]^-{\zeta\otimes\id} \ar[d]_{\id\otimes\iota} & L_\zeta \\
      \Omega^n k \otimes \Omega^n k \ar[r]_-{\Pp_1(\zeta)} & \Omega k \ar[u]_{\eta} } 
\]
\end{lemma}
\begin{proof}[Proof of Theorem~\ref{productiveTheorem}]
We assume that $[\zeta]\neq 0$. If $\Pp_1(\zeta)$ is divisible by $\zeta$, then there is a map $\alpha:\Omega^n k\otimes \Omega^n k\rightarrow \Omega^{n+1}k$ such that $\Pp_1(\zeta)=\zeta\alpha$. But then $\eta\Pp_1(\zeta)=\eta\zeta\alpha=0$ because $\eta\zeta=0$. By Proposition~\ref{commutativeSquare} we get that $\zeta$ is productive.

Conversely, suppose that $\zeta$ is productive, so that $\eta\Pp_1(\zeta)\cdot (\id\otimes\iota)=0$ by Proposition~\ref{commutativeSquare}. Since the triangle
\[ \Omega^n k \otimes L_\zeta \xrightarrow{\id\otimes\iota} \Omega^n k \otimes \Omega^n k \xrightarrow{\id\otimes \zeta} \Omega^n k \]
is exact, we get that $\eta\Pp_1(\zeta) = \lambda\cdot(\id\otimes\zeta)$ for some map $\lambda:\Omega^n k \rightarrow L_\zeta$. When we apply the homological functor $\underline{\Hom}_{kG}(\Omega^n k, -)$ to the triangle \eqref{zetaTriangle}, we get a long exact sequence
\[ \underline{\Hom}_{kG}(\Omega^n k, \Omega k) \xrightarrow{\eta_*} \underline{\Hom}_{kG}(\Omega^n k, L_\zeta) \xrightarrow{\iota_*} 
   \underline{\Hom}_{kG}(\Omega^n k, \Omega^n k ) \xrightarrow{\zeta_*} \underline{\Hom}_{kG}(\Omega^n k, k).
\]
Here $\zeta_*$ can be viewed as $\zeta\cdot:\hat{H}^0(G)\rightarrow \hat{H}^n(G)$ which is injective because the class $[\zeta]$ is non-zero. By exactness, $\iota_*=0$ and $\eta_*$ is surjective. In particular, $\lambda = \eta \rho$ for some map $\rho:\Omega^n k \rightarrow \Omega k$. Altogether we have that
\[ \eta(\Pp_1(\zeta) - \rho(\id\otimes\zeta)) = \lambda(\id\otimes \zeta) - \eta\rho(\id\otimes \zeta) = 0, \]
and therefore $\Pp_1(\zeta)-\rho(\id\otimes\zeta)=\zeta\sigma$ for some map $\sigma:\Omega^n k\otimes \Omega^n k \rightarrow \Omega^{n+1}k$. But then $\Pp_1(\zeta) = [\rho][\zeta] + [\zeta][\sigma]$, so that $\Pp_1(\zeta)$ is divisible by $[\zeta]$. 
\end{proof}

\begin{remark}
Before we start proving that the diagram commutes, let us draw some analogies to the topological world. Let us define $k/\zeta$ to be some choice of cone of $\zeta:\Omega^n k\rightarrow k$. Then the commutative square of Proposition~\ref{commutativeSquare} is a shift of the diagram on the left-hand side:
\[ \xymatrix{ 
 \Omega^n k \otimes k/\zeta \ar[r]^-{\zeta} \ar[d] & k/\zeta & &   S/2 \ar[r]^{2} \ar[d]_{\text{pinch}} & S/2 \\
\Omega^n k\otimes \Omega^{n-1} k \ar[r]_-{\Sq_1(\zeta)} & k \ar[u] & & \Sigma S \ar[r]_{\boldsymbol{\eta}} & S \ar[u]_{\text{incl}} } \] 
Note the similarity to the topological situation on the right-hand side, which takes place in the stable homotopy category. Here, $S$ denotes the sphere spectrum, $S/2$ is the mod-$2$-Moore spectrum, a cone of multiplication by $2$ on $S$, and $\boldsymbol{\eta} = \Sq_1(2)$ is the Hopf map. 
\end{remark}

The rest of this section is devoted to the proof of Proposition~\ref{commutativeSquare}. Let $p=2$, and let $G$ be a finite group. Let $P$ be a complete projective resolution of the trivial $kG$-module $k$, and let this resolution define the modules $\Omega^n k$.
\begin{lemma}\label{pponeRepresentative} 
 There is an unstably commutative diagram
\[ \xymatrix{
     \Omega^n k\otimes \Omega^n k \ar[r] \ar@{=}[d] & Q \ar[r] \ar[d] & k \ar@{=}[d] \\
     \Omega^n k\otimes \Omega^n k \ar[r]_{1+T} & \Omega^n k\otimes \Omega^n k \ar[r]_-{\zeta^{\otimes 2}} & k } 
\]
in which $Q$ is projective, and the upper row is a complex representing $\Pp_1(\zeta)\in\hat{H}^{2n-1}(G)$.
\end{lemma}
\begin{proof}
We use the operad $\C$ given in \S\ref{tateoperaddefinition} for the definition of $\Pp_1(\zeta)$. Choose a morphism of augmented $\Sigma_2$-chain complexes $\Delta:W\rightarrow \C(2)$, where $W$ is the standard free resolution of $k$ as trivial $\Sigma_2$-module. When we consider $\Omega^n k$ as a chain complex concentrated in degree $0$, then the differential of $P$ induces a chain map $\pi_n:P\rightarrow\Omega^n k$ of degree~$n$. Let $\Delta_i=\Delta(e_i)$, where $e_i$ is the generator of $W_i$. We get a commutative diagram
\[ \xymatrix{ P \ar[r]^{\Delta_0} \ar[d]^{\pi_{2n}} & P^{\boxtimes 2} \ar[d]^{\pi_n^{\boxtimes 2}} \ar[r]^{1+T} & P^{\boxtimes 2} \ar[d]^{\pi_n^{\boxtimes 2}} \\
             \Omega^{2n} k \ar[r]_-{\psi} &  (\Omega^n k)^{\otimes 2} \ar[r]_{1+T} & (\Omega^n k)^{\otimes 2} }
\]
where $\psi$ is a stable equivalence and the upper row equals $d\Delta_1$. Since $\pi_n$ is a chain map, this diagram restricts to the following commutative diagram in dimension $2n$:
\[ \xymatrix{ P_{2n} \ar[r]^{\partial} \ar[d]^{\pi_{2n}} & P_{2n-1} \ar[r]^{\Delta_1} & (P^{\boxtimes 2})_{2n} \ar[d]^{\pi_n^{\boxtimes 2}} \\
              \Omega^{2n} k \ar[r]_-{\psi} &  (\Omega^n k)^{\otimes 2} \ar[r]_{1+T} & (\Omega^n k)^{\otimes 2} }
\]
We define $\lambda = \pi_n^{\boxtimes 2}\circ \Delta_1$; then $\zeta^{\otimes 2}\circ \lambda:P_{2n-1}\rightarrow k$ represents $\Pp_1(\zeta)$ by definition.
Since $\psi$ is a stable equivalence, we can choose a map $\omega:(\Omega^n k)^{\otimes 2}\rightarrow \Omega^{2n} k$ such that $\psi\omega-\id$ factors as 
\[ \psi\omega-\id:(\Omega^n k)^{\otimes 2}\xrightarrow{\alpha} Q \xrightarrow{\beta} (\Omega^n k)^{\otimes 2} \]
for some projective module $R$. Then we have a commutative diagram
\[ \xymatrix@C=45pt{ 
   \Omega^{2n} k \ar[r]^{\text{incl}} & P_{2n-1} \ar[r]^{\zeta^{\otimes 2}\circ\lambda} & k \\
   \Omega^n k\otimes\Omega^n k \ar[r]^{\smatrix{\text{incl}\circ \omega \\ \alpha}} \ar@{=}[d] \ar[u]^{\omega} & 
        P_{2n-1}\oplus R \ar[d]_{\smatrix{\lambda, & (1+T)\circ \beta}} \ar[r]^-{\smatrix{\zeta^{\otimes 2}\circ \lambda, & 0}} \ar[u]^{\smatrix{\id, & 0}} & 
        k \ar@{=}[d] \ar@{=}[u] \\
   \Omega^n k\otimes \Omega^n k \ar[r]_{1+T} & \Omega^n k\otimes \Omega^n k \ar[r]_-{\zeta^{\otimes 2}} & k }
\]
proving the claim.
\end{proof}

\begin{proof}[Proof of Proposition~\ref{commutativeSquare}]
Define the map $\kappa:\Omega^n k \otimes \Omega^n k \rightarrow L_\zeta$ by $a\otimes b \mapsto \zeta(a)b+\zeta(b)a$, then the upper left triangle in the diagram
\[ \xymatrix{
      \Omega^n k \otimes L_\zeta \ar[r]^-{\zeta\otimes\id} \ar[d]_{\id\otimes\iota} & L_\zeta \\
      \Omega^n k \otimes \Omega^n k \ar[ur]^{\kappa} \ar[r]_-{\Pp_1(\zeta)} & \Omega k \ar[u]_{\eta} } 
\]
commutes, and we want to show that the bottom right triangle also commutes. To do so, we extend the diagram of Proposition~\ref{pponeRepresentative} as follows:
\[ \xymatrix{
     \Omega^n k\otimes \Omega^n k \ar[r] \ar@{=}[d] & Q \ar[r] \ar[d] & k \ar@{=}[d] \\
     \Omega^n k\otimes \Omega^n k \ar[r]_{1+T} \ar[d]_{\kappa} & \Omega^n k\otimes \Omega^n k \ar[r]_-{\zeta^{\otimes 2}} \ar[d]_{\zeta\otimes\id} & k \ar@{=}[d]  \\
     L_\zeta \ar[r]_{\iota} & \Omega^n k \ar[r]_{\zeta} & k} 
\]
The bottom row is an extension representing $[\eta]\in\Ext^1_{kG}(k,L_\zeta)$, so we are done by Proposition~\ref{triproductlemma}.
\end{proof}

\bibliographystyle{plain}

\bigskip

\end{document}